\documentclass[11pt,letterpaper]{amsart}
\usepackage{mathrsfs}
\usepackage{geometry}

\usepackage{graphicx}
\usepackage{amssymb}
\usepackage{epstopdf}
\usepackage{amsmath,amscd}
\usepackage{amsthm}
\usepackage{float}
\usepackage{enumitem}
\usepackage{setspace}
\usepackage{cleveref}
\usepackage{url,verbatim}
\usepackage{breakurl}
\usepackage{color}

\usepackage{setspace}

\usepackage[numbers,sort&compress]{natbib}

\usepackage[caption = false]{subfig}
\theoremstyle{plain} \textwidth=430pt \textheight=630pt

\newtheorem{theorem}{Theorem}[section]

\newtheorem{lemma}[theorem]{Lemma}

\newtheorem{proposition}[theorem]{Proposition}

\newtheorem{corollary}[theorem]{Corollary}

\newcommand\ZZ{{\mathbb Z}}
\newcommand\RR{{\mathbb R}}
\newcommand\sA{{\mathcal A}}
\newcommand\sB{{\mathcal B}}
\newcommand\LL{{\mathbb L}}

\newcommand\sE{{\mathcal E}}
\newcommand\sF{{\mathcal F}}

\newcommand\sR{{\mathcal R}}
\newcommand\sH{{\mathcal H}}
\newcommand\sN{{\mathcal N}}

\newcommand\df{\textbf}

\renewenvironment{proof}{{\bfseries \noindent Proof.\ }}{\qed}

\title{Constrained Percolation in Two Dimensions
}
\author{Alexander E. Holroyd}
\address{Microsoft Research, Redmond, WA 98052, USA}

\email{holroyd@microsoft.com}
\urladdr{\url{http://aeholroyd.org}}

\author{Zhongyang Li}
\address{Department of Mathematics,
University of Connecticut,
Storrs, Connecticut 06269-3009, USA}
\email{zhongyang.li@uconn.edu}
\urladdr{\url{http://www.math.uconn.edu/~zhongyang/}}

\begin{document}
\maketitle

%

\begin{abstract}
We prove absence of infinite clusters and contours in a
class of critical constrained percolation models on the
square lattice.  The percolation configuration is assumed
to satisfy certain hard local constraints, but only weak
symmetry and ergodicity conditions are imposed on its
law. The proofs use new combinatorial techniques
exploiting planar duality.

Applications include absence of infinite clusters of
diagonal edges for critical dimer models on the
square-octagon lattice, as well as absence of infinite
contours and infinite clusters for critical XOR Ising
models on the square grid.  We also prove that there
exists at most one infinite contour for high-temperature
XOR Ising models, and no infinite contour for
low-temperature XOR Ising model.
\end{abstract}

\section{Introduction}\label{cp}

\subsection{Background}

A central question in percolation and statistical physics
models is when there exists an \emph{infinite cluster};
that is, an infinite connected component of elements having
the same state.  The onset of infinite clusters as a model
parameter is varied may often be taken as a defining
characteristic of a phase transition or critical point.

A key problem is to determine whether infinite clusters are
present \emph{at} the critical point. In particular, this
question remains open in the archetypal case of independent
bond percolation on the hypercubic lattice in dimensions
$3\leq d\leq 10$ (see \cite{he80,grgP,HS94,fh15} for
details). Far more is known about certain two-dimensional
models, culminating in spectacularly detailed understanding
of the critical phase, via Schramm-Loewner Evolution
\cite{smirnov,smirnov-werner,lawler,werner,lsw02}. In many
settings, a necessary precondition for analysis of this
kind is \emph{self-duality}. One then expects the phase
transition point to coincide with the self-dual point,
which is the point where the primal and dual models have
equal parameters. However, proving this is often difficult,
since one must rule out coexistence of primal and dual
infinite clusters. This question was open for 20 years
(between \cite{ha60} and \cite{he80}) for bond percolation
on the square lattice. A further important circle of
questions concerns the number of infinite clusters.
Uniqueness of the infinite cluster for independent
percolation on the hypercubic lattice was also open for
many years, before the proof in \cite{akn} (later
simplified and generalized in \cite{bk89}). The central
conjecture in percolation on general transitive graphs is
that non-amenability of the graph is equivalent to
existence of a non-uniqueness phase \cite{beyond}.

In this article we address questions of existence,
coexistence, and uniqueness of infinite clusters for a
class of \emph{constrained} percolation models in two
dimensions. By this we mean that the configuration is
restricted to lie in a subset of the sample space where
certain hard local constraints are satisfied.  Subject to
this restriction, the probability measures that we consider
are very general.  Our main results require only very weak
symmetry and ergodicity conditions.  In particular, we do
not assume stochastic monotonicity or correlation
conditions such as the FKG inequality.

Many standard statistical physics models can be interpreted as constrained
percolation models.  Examples include the dimer model, or perfect matching
model (in which a configuration is a subset of edges in which each vertex has
exactly one incident edge present \cite{rkle}); the 1-2 model (where each
vertex has one or two incident present edges \cite{GrLreview}); the 6-vertex
model (where configurations are edge orientations on a degree-4 graph in
which each vertex has in- and out-degree 2 \cite{Bax08}); and some general
vertex models that can be transformed to dimer models on decorated graphs via
the holographic algorithm (\cite{Val, CL, SB, ZLlv}).  We will give
applications of our main results to dimer and XOR Ising models.




Phase transitions of certain constrained percolation
problems have been studied extensively. See, for example,
\cite{KOS06} for the dimer model, and \cite{GrL15} for the
1-2 model. The integrability properties of these
constrained percolation problems make it possible to
exactly compute finite-dimensional distributions and
correlation functions. The critical parameter, i.e.\ the
parameter where discontinuity of a certain correlation
function is observed, can often be computed as the solution
of an explicit algebraic equation.

Although there are many results describing the phase
transitions of constrained percolation problems by a
microscopic observable, e.g.\ spin-spin correlation, up to
now, very few papers study the phase transitions of
constrained percolation models from a macroscopic
perspective, e.g.\ the existence of an infinite cluster.
Even though sometimes we know that a phase transition
exists with respect to a macroscopic observable, the exact
value of the critical parameter is unknown \cite{ZLejp},
and it is not known if the critical parameter in the
macroscopic sense coincides with the critical parameter in
the microscopic sense, except for very few special models,
such as the 2-dimensional Ising model
\cite{ABF87,ZLSC,ZL12}. One difficulty for constrained
percolation problems is that there is often no stochastic
monotonicity; see also \cite{She05,PM13, HM09}.


\subsection{Constrained Percolation}

In this paper, we study a class of constrained site
percolation problems on the 2-dimensional square lattice
$G=(\ZZ^2,E)$. The vertex set $\ZZ^2$ consists of all
points $(m,n)$ with integer coordinates. Two vertices
$(m,n)$ and $(m',n')$ are joined by an edge in $E$ if and
only if $|m-m'|+|n-n'|=1$.

Each face of $G$ is a unit square. We say that two faces of $G$ are
\df{adjacent} if they share an edge. Let $f$ be a face of $G$, and let
$(m_f,n_f)$ be the coordinate of the vertex at the lower left corner of $f$.
We color $f$ \df{white} if $m_f+n_f$ is odd. If $m_f+n_f$ is even, we color
$f$ \df{black}.

We consider site percolation on $G$, i.e.\ the state space is $\{0,1\}^{\ZZ^2}$. We call an element $\omega$ of $\{0,1\}^{\ZZ^2}$ a \df{configuration}, and we call $\omega(v)\in\{0,1\}$ the \df{state} of $v\in \ZZ^2$.  We impose the following constraint on site configurations.
\begin{itemize}
\item Around each black face, there are 6 allowed configurations $(0000)$, $(1111)$, $(0011)$, $(1100)$, $(0110)$, $(1001)$, where the digits from the left to the right correspond to vertices in clockwise order around the black square, starting from the lower left corner.
\end{itemize}

Note that in the unconstrained case, around each black
square, there are 16 different configurations, only 6 of
which are allowed in the constrained case. See Figure
\ref{lcc} for local configurations of the constrained
percolation around a black square.

\begin{figure}
\subfloat[0000]{\includegraphics[width=.16\textwidth]{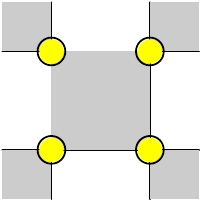}}\qquad\qquad\qquad\qquad
\subfloat[0011]{\includegraphics[width = .16\textwidth]{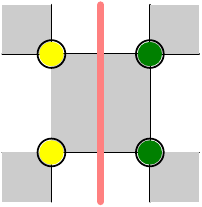}}\qquad\qquad\qquad\qquad
\subfloat[0110]{\includegraphics[width = .16\textwidth]{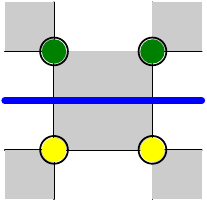}}\\
\subfloat[1111]{\includegraphics[width = .16\textwidth]{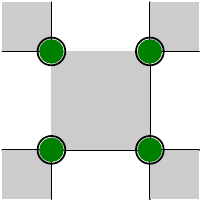}} \qquad\qquad\qquad\qquad
\subfloat[1100]{\includegraphics[width = .16\textwidth]{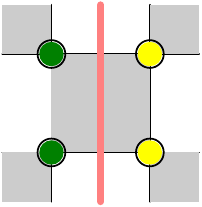}}\qquad\qquad\qquad\qquad
\subfloat[1001]{\includegraphics[width = .16\textwidth]{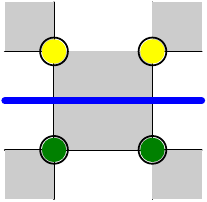}}
\caption{Local configurations of the constrained percolation around a black square. Red and blue lines mark contours separating 0's and 1's (in $\LL_1$ and $\LL_2$ respectively). Yellow (resp.\ green) disks represent 0's (resp.\ 1's).}
\label{lcc}
\end{figure}
Let $\Omega\subset \{0,1\}^{\ZZ^2}$ be the set of all configurations
satisfying the constraint above. Let $\mu$ be a probability measure on
$\Omega$.  We will be interested in such measures $\mu$ satisfying the
following conditions.

\begin{enumerate}[label=({A}{\arabic*})]
\item $\mu$ is $\sH$-translation-invariant, where $\sH$ is the subgroup of $\ZZ\times\ZZ$ generated by $(1,1)$ and $(1,-1)$.
\item $\mu$ is $2\ZZ\times 2\ZZ$-ergodic; i.e., any $2\ZZ\times 2\ZZ$-translation-invariant event has probability $0$ or $1$ under $\mu$.
\item $\mu$ is symmetric under exchanging $0$ and $1$; i.e.\ writing
    $\theta:\Omega\rightarrow\Omega$ for the map defined by
    $\theta(\omega)(v)=1-\omega(v)$ for each $v\in \ZZ^2$, the measure
    $\mu$ is invariant under $\theta$, that is, $\mu(A)=\mu(\theta(A))$
    for every event $A$.
\end{enumerate}

We will also consider the following variations of (A1) and (A2) (one weaker
and one stronger). Let $k$ be a positive integer.
\begin{enumerate}[label=({Ak}{\arabic*})]
\item $\mu$ is $2k\ZZ\times 2k\ZZ$-translation-invariant.
\item $\mu$ is $2k\ZZ\times 2k\ZZ$-ergodic; i.e., any $2k\ZZ\times 2k\ZZ$-translation-invariant event has probability $0$ or $1$ under $\mu$.
\end{enumerate}
Note for $k\geq 1$ we have
$$
(A1)\Rightarrow (Ak1);\qquad
(Ak2)\Rightarrow(A2),
$$
where ``$\Rightarrow$'' means ``implies''.

\subsection{Contours and Clusters}\label{cot}

Let $V\subseteq \ZZ^2$ be a set of vertices in $\ZZ^2$. We say that $V$ is
\df{connected} if it induces a connected subgraph of $G$.

Let $\omega\in \Omega$. A \df{cluster} of $\omega$ is a maximal connected set
of vertices of $G$ in which every vertex has the same state. If all the
vertices in a cluster have the state 0 (resp.\ 1), we call the cluster a
0-cluster (resp.\ 1-cluster). A cluster may be finite or infinite.  Here is
our first main result.

\begin{theorem}\label{m1}
Let $\mu$ be a probability measure on the constrained
percolation state space $\Omega$, satisfying $(A1)-(A3)$.
Let $\sA$ be the event that the number of infinite clusters
is nonzero and finite.  Then
\begin{eqnarray*}
\mu(\sA)=0.
\end{eqnarray*}
\end{theorem}

Note that \Cref{m1} requires no assumptions of stochastic monotonicity, correlation inequalities, or rotation-invariance. See \cite{crpr76,Hig93,Berg08} for related results requiring stochastic monotonicity and rotation-invariance.

The conclusion of \Cref{m1} does not in general hold for
\emph{unconstrained} percolation meausres. Here is an
example.  Let $X=(X_m)_{m\in\ZZ}$ and $Y=(Y_n)_{n\in\ZZ}$
be independent families of i.i.d.\ Bernoulli random
variables with parameter $1/2$.  Define a \df{run} of $X$
(resp.\ $Y$) to be a maximal nonempty interval of $\ZZ$ on
which the corresponding variables are all equal. Define a
\df{run rectangle} to be a set of the form $I\times
J\subset\ZZ^2$ where $I$ is a run of $X$ and $J$ is a run
of $Y$.  Note that the run rectangles partition $\ZZ^2$.
Call a run rectangle a \df{site rectangle} if both $I$ and
$J$ are runs of 1s, and a \df{bond rectangle} if exactly
one of them is a run of $1$s. Let $\mathcal{G}$ be the
(random) graph whose vertex set is the set of site
rectangles, whose edge set is the set of all bond
rectangles, and where a vertex and an edge are incident if
some site of one is adjacent in $G$ to some site of the
other. It is easily seen that $\mathcal{G}$ is isomorphic
to $G$ a.s. Now take a uniform spanning tree of
$\mathcal{G}$ (conditional on $X$ and $Y$).  Finally,
assign a vertex of $G$ the value $1$ if it lies in a site
rectangle or it lies in a bond rectangle whose edge is
included in the spanning tree; otherwise assign value $0$.
It is straightforward to check that the resulting random
configuration on $G$ is $\sH$-translation-invariant and
symmetric. Moreover, a.s.\ there exist both an infinite
0-cluster and an infinite 1-cluster. Indeed, the measure is
$2\ZZ\times 2\ZZ$ ergodic as well. (This can be checked
from mixing properties of the uniform spanning tree.)


We also consider contours separating clusters. For this purpose, we introduce
two auxiliary square grids, $\LL_1$ and $\LL_2$, whose vertices are located
at centers of white faces of the original square grid $G$. The primal (resp.\
dual) auxiliary square grid $\LL_1$ (resp.\ $\LL_2$) has vertices located at
points $(m-\frac{1}{2},n+\frac{1}{2})$ of the plane, in which both $m$ and
$n$ are even (resp.\ odd). Two vertices $(m-\frac{1}{2},n+\frac{1}{2})$,
$(m'-\frac{1}{2},n'+\frac{1}{2})$ of $\LL_1$ (resp.\ $\LL_2$) are joined by
an edge of $\LL_1$ (resp.\ $\LL_2$) if and only if $|m-m'|+|n-n'|=2$.
Evidently each face of $\LL_1$ or $\LL_2$ is a square of side length 2. See
Figure \ref{gll}.

\begin{figure}
\includegraphics{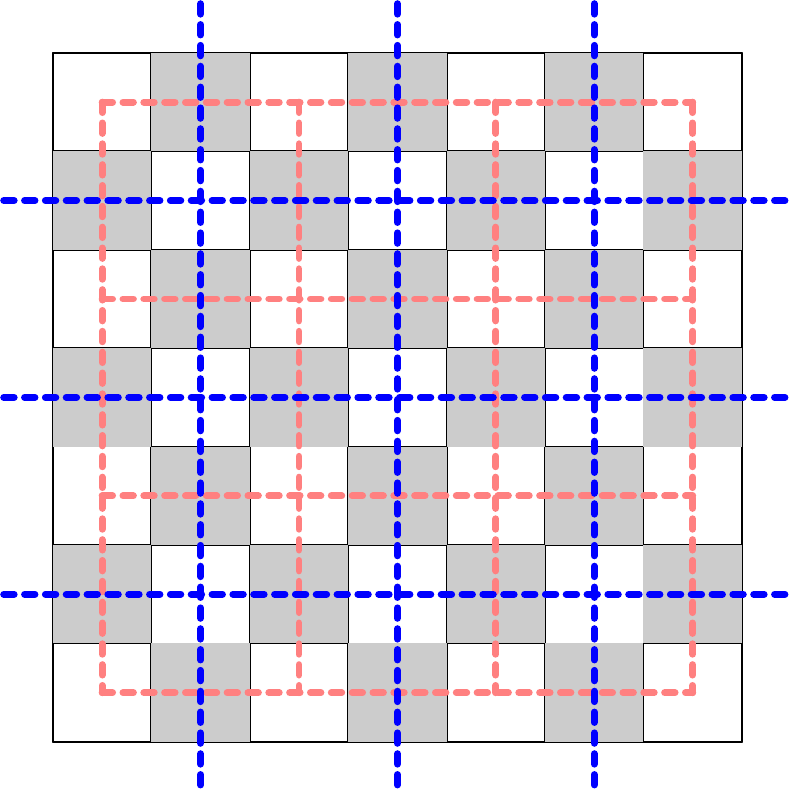}
\caption{Graphs $G$, $\LL_1$, and $\LL_2$. Solid lines represent $G$. Red dashed lines represent $\LL_1$. Blue dashed lines represent $\LL_2$.}\label{gll}
\end{figure}

Moreover, each black face $F$ of $G$ is crossed by an edge $e_1$ of $\LL_1$
(in the sense that $F$ and $e_1$ share a center). Similarly, $F$ is crossed
by an edge $e_2$ of $\LL_2$. The two edges of $\LL_1$ and $\LL_2$ that cross
the same black face of $G$ are perpendicular to each other.  Each
configuration in $\omega\in\Omega$ corresponds to a configuration in
$\phi(\omega)\in\{0,1\}^{E(\LL_1)\cup E(\LL_2)}$, where $E(\LL_1)$ (resp.\
$E(\LL_2)$) is the edge set of $\LL_1$ (resp.\ $\LL_2$), as follows.  For
each black face $F$ of $G$,  if the configuration around $F$ is $(0000)$ or
$(1111)$, then we set $\phi(\omega)(e_1)=\phi(\omega)(e_2)=0$, for the two
edges $e_1\in E(\LL_1)$ and $e_2\in E(\LL_2)$ that cross $F$. If the
configuration around $F$ is $(1001)$ or $(0110)$ (so that the two upper
vertices have one state, and the two lower vertices have the other), then we
let $\phi(\omega)$ take value 1 on the horizontal edge ($e_1$ or $e_2$), and
0 on the vertical edge. Similarly in the cases $(0011)$ and $(1100)$, we set
$\phi(\omega)$ to be 1 on the vertical edge and 0 on the horizontal edge. We
say an edge $e\in E(\LL_1)\cup E(\LL_2)$ is \df{present} (resp.\ \df{absent})
if it has state $1$ (resp.\ 0). See Figure \ref{spc}.

\begin{figure}
\includegraphics{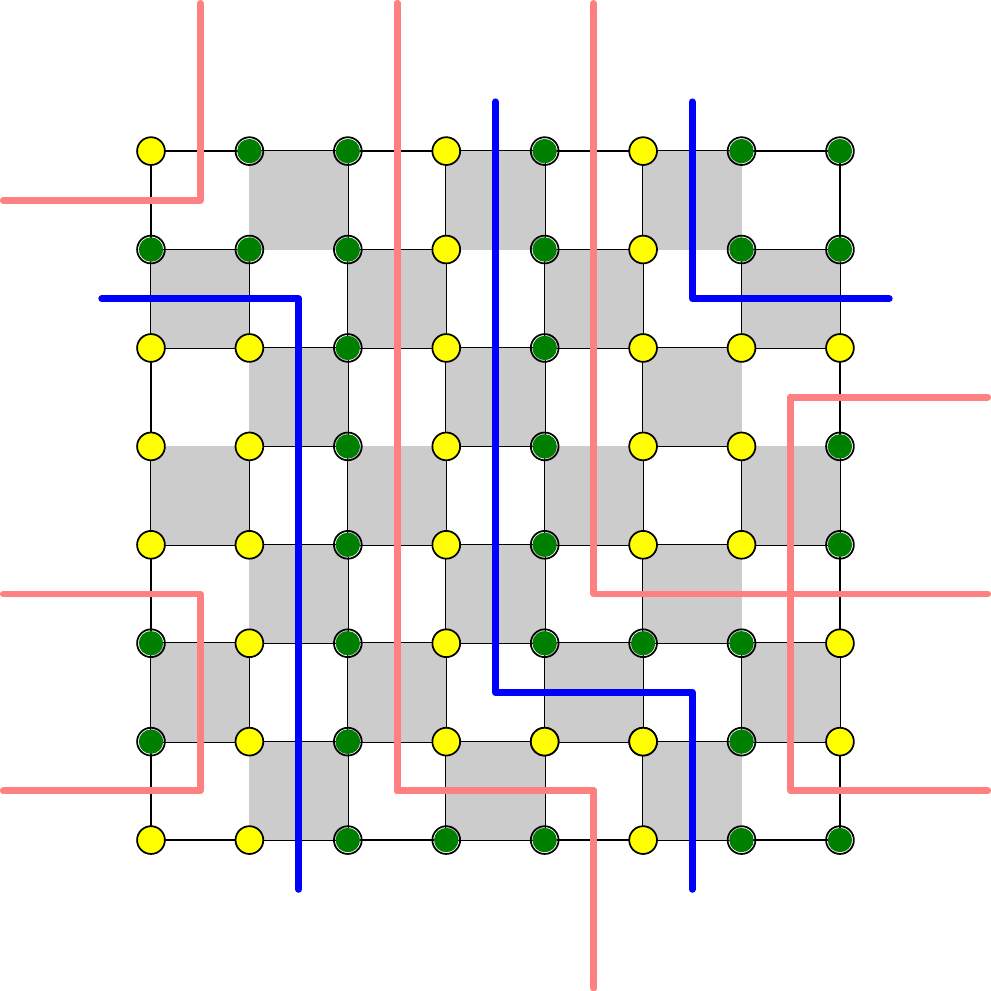}
\caption{Constrained percolation configuration and associated
contour configuration. Red lines represent primal contours.
Blue lines represent dual contours.
Green and yellow discs represent the two states of vertices. }\label{spc}
\end{figure}

The present edges of $\LL_1$ and $\LL_2$ are boundaries separating the state
0 and the state 1. It is easily verified that each vertex of $\LL_1$ or
$\LL_2$ has an even number of incident present edges. Moreover, the present
edges of $\LL_1$ and $\LL_2$ can never cross.

The image of $\Omega$ under the map $\phi$ is a subset of $\{0,1\}^{E(\LL_1)\cup E(\LL_2)}$. We call elements of this image \df{contour configurations}. A configuration in $\{0,1\}^{E(\LL_1)\cup E(\LL_2)}$ is a contour configuration if and only if
\begin{enumerate}
\item \label{c1} each vertex in $\LL_1$ or $\LL_2$ has an even number of incident present edges;
\item \label{c2} present edges in $E(\LL_1)$ and present edges in $E(\LL_2)$ do not cross.
\end{enumerate}
We use $\Phi$ to denote the set of all contour
configurations. The map $\phi:\Omega\rightarrow\Phi$ is
surjective and 2-to-1.  Specifically, the configurations
$\omega$ and $\theta(\omega)$ (but no others) have the same
image under $\phi$, where $\theta$ is the $0/1$ exchange
map defined in (A3).

Let $\omega\in \Omega$, and $\phi=\phi(\omega)\in\Phi$. Each connected
component of present edges in $\phi$ is called a \df{contour} of $\phi$.
Since present primal edges and dual primal edges do not cross in a contour
configuration, either all the edges in a contour are primal edges (edges of
$\LL_1$), or all the edges in a contour are dual edges (edges of $\LL_2$). A
contour is a \df{primal} (resp.\ \df{dual}) \df{contour} if it consists of
edges of $\LL_1$ (resp.\ edges of $\LL_2$). A (primal or dual) contour is
called finite (resp.\ infinite) if it consists of finitely many (resp.\
infinitely many) edges. Note that a contour may have 4 edges sharing a
vertex; see Figure \ref{spc}.

Let $D$ be a cluster of $\omega\in \Omega$, and let $C$ be a contour of
$\phi(\omega)\in\Phi$. We say $C$ is \df{incident} to $D$ if there exists
$e\in C\subseteq E(\LL_1)\cup E(\LL_2)$, and a vertex $v\in D\subseteq
\ZZ^2$, such that $v$ is at Euclidean distance $\frac{1}{2}$ from the center
point of $e$.



Let $\mu$ be a probability measure on $\Omega$. Note that $\mu$ induces a probability measure $\nu$ on contour configurations in $\Phi$, via the map $\phi$. We consider a random configuration in $\Omega$ with law $\mu$. Then the image $\phi$ is an associated contour configuration with law $\nu$.

 Let $\nu_1$ (resp.\ $\nu_2$) be the corresponding marginal distribution of $\nu$ on bond configurations of $\LL_1$ (resp.\ $\LL_2$). Let $\Phi_1$ (resp.\ $\Phi_2$) be the state space consisting of bond configurations of $\LL_1$ (resp.\ $\LL_2$) satisfying the condition that each vertex has an even number of incident present edges. In some cases, we may wish to assume the following.

 \begin{enumerate}[label=({A}{\arabic*})]
\setcounter{enumi}{3}
\item $\nu_1$ has \df{finite energy} in the following sense: let $S$ be a
    face of $\LL_1$, and $\partial S\subset E(\LL_1)$ be the set of four
    sides of the square $S$. Let $\phi\in \Phi_1$. Define $\phi_S$ to be
    the configuration obtained by switching the states of each element of
    $\partial S$, i.e.\ $\phi_S(e)=1-\phi(e)$ if $e\in\partial S$, and
    $\phi_S(e)=\phi(e)$ otherwise; see Figure \ref{ccl}. Let $E$ be an
    event, and define
\begin{eqnarray}
E_S=\{\phi_S:\phi\in E\}.\label{ef}
\end{eqnarray}
Then $\nu_1(E_S)>0$ whenever $\nu_1(E)>0$.
\end{enumerate}

\begin{figure}
 \centering
  \includegraphics[width=0.3\textwidth]{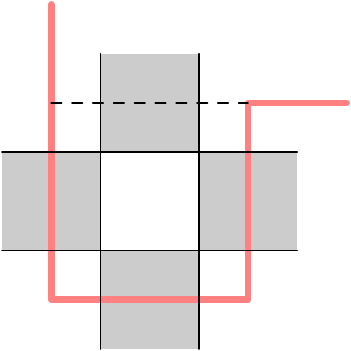}\qquad\qquad \includegraphics[width=0.3\textwidth]{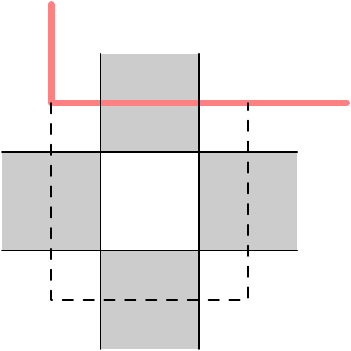}
 \caption{Modifying a contour configuration by flipping a unit square.
   The finite energy condition states that the law of the contour configuration after such a
   change is absolutely continuous with respect to the original.}\label{ccl}
 \end{figure}

Note that, for each $\phi\in \Phi_1$, the corresponding $\phi_S$ defined in Assumption (A4)
is still in $\Phi_1$.

\begin{samepage}
\begin{theorem}\label{m2}Let $\mu$ be a probability measure on $\Omega$,
and consider the corresponding contour configurations as defined above. Under
Assumptions $(A1)$--$(A4)$:
\begin{enumerate}
\item $\mu$-a.s.\ there are neither infinite primal contours nor infinite dual contours.
\item $\mu$-a.s.\ there are no infinite clusters.
\end{enumerate}
\end{theorem}
\end{samepage}

Let $\xi$ be the (0- or 1-) cluster including the origin in a random constrained percolation configuration. We define the mean cluster size $\chi$ as follows
\begin{eqnarray}
\chi=\mathbb{E}|\xi|,\label{mcs}
\end{eqnarray}
where $|\xi|$ is the number of vertices in the cluster $\xi$. By Proposition 1 of \cite{Rus78}, in any translation-invariant measure on $\{0,1\}^{\ZZ^2}$ that has no infinite 0- or 1-clusters satisfies $\chi=\infty$. Therefore \Cref{m2} implies that $\chi=\infty$ for any probability measure on $\Omega$ satisfying (A1)-(A4).

Note that Assumption (A4) is important for the conclusion
of Theorem \ref{m2}. In fact, if Assumption (A4) does not
hold, it is possible that there exists more than one
infinite cluster with positive probability. Consider, for
example, a distribution of  constrained percolation
configurations on $\ZZ^2$, such that each row of $\ZZ^2$ is
either all 0's with probability $\frac{1}{2}$, or all 1's
with probability $\frac{1}{2}$, and the configurations on
different rows are independent. This distribution satisfies
Assumptions (A1)--(A3), but not (A4) (and it is not ergodic
under the group of horizontal translations). With
probability 1 there exist infinitely many infinite clusters
(indeed, $*$-clusters) under such a distribution.

Under the same Assumptions (A1)--(A4), we have a stronger
conclusion. In order to state the conclusion, let $\phi_1$
(resp.\ $\phi_2$) be a contour configuration on $\LL_1$
(resp.\ $\LL_2$). Let $G\setminus\phi_1$ (resp.\
$G\setminus\phi_2$) be the graph obtained from $G$ by
removing every edge that is crossed by a present edge of
$\phi_1$ (resp.\ $\phi_2$).

\begin{theorem}\label{m3}Let $\mu$ be a probability measure on $\Omega$ satisfying the Assumptions $(A1)$--$(A4)$. Let $\nu_1$  be the corresponding marginal distribution on bond configurations in $\Phi_1$. Let $\phi_1$ be the union of all primal contours. Then $\nu_1$-a.s.\ $G\setminus \phi_1$ has no infinite components.
\end{theorem}

We also have some results on contours without the symmetry
assumption (A3).

\begin{theorem}\label{m4}Let $\mu$ be a probability measure on the constrained percolation state space $\Omega$, satisfying  Assumptions $(Ak1)$,$(Ak2)$,$(A4)$, for some positive integer $k$. Then $\mu$-a.s.\ there is at most one infinite primal contour.
\end{theorem}

Finally, a random contour configuration $\phi\in \Phi_2$ induces a random
site configuration $\rho$ on $\{0,1\}^{V(\LL_1)}$ as follows. Fix a vertex
$v_0\in V(\LL_1)$, and assume that $\rho(v_0)$ takes value 1 with probability
$\frac{1}{2}$, and takes value $0$ with probability $\frac{1}{2}$,
independent of $\phi$. For any two adjacent vertices $v,w\in V(\LL_1)$, let
$\rho(v)\neq \rho(w)$ if and only if the edge $\langle v,w\rangle$ crosses a
present edge in $\phi$. Let $\lambda_1$ be the measure on
$\{0,1\}^{V(\LL_1)}$, induced by $\nu_2$ in the way described above. We
introduce the following new assumption.
 \begin{enumerate}[label=({A}{\arabic*})]
\setcounter{enumi}{4}
\item $\lambda_1$ is $2\ZZ\times 2\ZZ$-ergodic.
\end{enumerate}

\begin{theorem}\label{m5}Let $\mu$ be a probability measure on the constrained percolation state space $\Omega$, satisfying $(Ak1)$, $(Ak2)$, $(A4)$, $(A5)$, for some positive integer $k$. Then $\mu$-a.s.\ there are no infinite primal contours.
\end{theorem}

\subsection{Examples of Constrained Percolation Measures}\label{ex}

We present some applications of Theorems \ref{m1}--\ref{m5}
to perfect matchings on the square-octagon lattice, as well
as the XOR Ising model on the square grid.  We will obtain
results about infinite clusters and infinite contours in
these well-known models.

\begin{figure}
 \centering
  \includegraphics[width=0.6\textwidth]{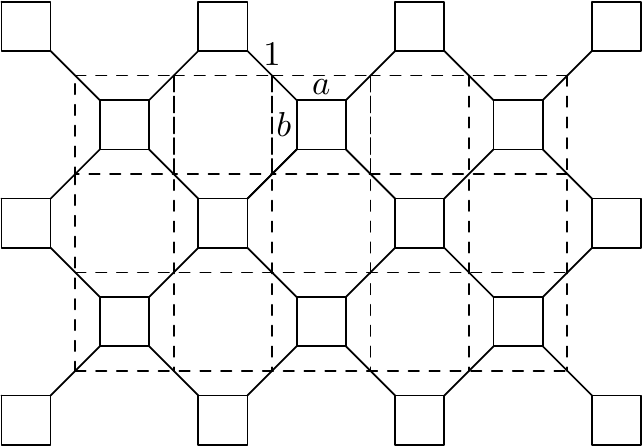}
 \caption{Square-octagon lattice: solid lines are edges of the square-octagon lattice; dashed lines are edges of $G$.}\label{sol}
 \end{figure}

Consider perfect matchings on the square-octagon lattice. See Figure
\ref{sol} for a picture of the square-octagon lattice. Each perfect matching
(or dimer configuration) on the square-octagon lattice is a subset of edges
such that each vertex is incident to exactly one edge in the subset. There
are two types of edges in the lattice: \df{Type-I edges} are edges of the
square faces, and \df{Type-II edges} are edges of the octagonal faces but not
of the square faces. (Type-II edges are diagonal lines in Figure \ref{sol}).

We now connect perfect matchings on the square-octagon
lattice with constrained percolation configurations on $G$.
Recall that $G=(\ZZ^2,E)$ is the square grid, whose faces
are unit squares. We place a vertex of $G$ at the midpoint
of each Type-II edge. A face of $G$ is constructed from the
midpoints of four Type-II edges around a square or the
midpoints of four Type-II edges around an octagon; see
Figure \ref{sol}. If a face of $G$ encloses a square face
of the square-octagon lattice, we color it black.  If a
face of $G$ is enclosed by an octagon face, we color it
white. 

Given a perfect matching of the square-octagon lattice, we may consider its restriction to the set of all Type-II edges. There is a bijective correspondence between such restrictions and site percolation configurations on $\ZZ^2$ in $\Omega$. Specifically,   a vertex of $\ZZ^2$ at the midpoint of a Type-II edge has state 1 in the constrained site configuration in $\Omega$ if and only if the Type-II edge is present in the perfect matching of the square-octagon lattice. It is easy to verify that this is indeed a bijection. A \df{present Type-II cluster} (resp.\ \df{absent Type-II cluster}) of a dimer configuration on the square-octagon lattice is a set of present (resp.\ absent) Type-II edges such that their midpoints form a 1-cluster (resp.\ 0-cluster) of the constrained percolation configuration, given
by the above bijection. Equivalently, Type-II clusters may be defined by considering two Type-II edges to be adjacent if some endpoint of one is adjacent to some endpoint of the other.

In order to define a probability measure for perfect matchings on the square-octagon lattice, we introduce edge weights. We assign weight 1 to each Type-II edge, and weight $w_e$ to the Type-I edge $e$. Assume that the edge weights of the square-octagon lattice satisfy the following conditions.
 \begin{enumerate}[label=({B}{\arabic*})]
\item The edge weights are $2\ZZ\times2 \ZZ$-translation-invariant.
\item If $e_1$, $e_2$ are two Type-I edges around the same square face, such that both $e_1$ and $e_2$ are horizontal, or both of them are vertical, then $w_{e_1}=w_{e_2}$.
\item If $e_1$, $e_2$ are two Type-I edges around the same square face, such that exactly one of $e_1$, $e_2$ is horizontal and the other is vertical, then $w_{e_1}^2+w_{e_2}^2=1$.
\end{enumerate}

The reason we assume (B1) is to define a $2\ZZ\times 2\ZZ$-translation-invariant measure. The reason we assume (B2) and (B3) is to define a measure for dimer configurations of the square-octagon lattice, which, under the connection described above to constrained percolation configurations in $\Omega$, will induce a probability measure on $\Omega$ satisfying the symmetry assumption (A4).

Under (B1)--(B3), the edge weights are described by two independent parameters. We may sometimes assume the stronger translation-invariance condition below, which reduces the parameters to one.
\begin{enumerate}[label=({B}{\arabic*})]
\setcounter{enumi}{3}
\item The edge weights are $\sH$-translation-invariant.
\end{enumerate}

In \cite{KOS06}, the authors define a probability measure for any bi-periodic, bipartite, 2-dimensional lattice. Specializing to our case, let $\mu_{n,D}$ be the probability measure of dimer configurations on a toroidal $n\times n$ square-octagon lattice $S_n$ (see \cite{KOS06} for details). Let $\mathcal{M}_n$ be the set of all perfect matchings on $S_n$, and let $M\in \mathcal{M}_n$ be dimer configuration, then
\begin{eqnarray}
\mu_{n,D}(M)=\frac{\prod_{e\in M}w_{e}}{\sum_{M\in \mathcal{M}_n}\prod_{e\in M}w_e},\label{mnd}
\end{eqnarray}
where $w_e$ is the weight of the edge $e$. As $n\rightarrow\infty$, $\mu_{n,D}$ converges weakly to a translation-invariant measure $\mu_{D}$ (see \cite{KOS06}).

Contours separating present Type-II clusters and absent Type-II clusters are
be defined to be the contours in the corresponding constrained percolation
model.

\begin{theorem}\label{dc}For given edge
weights satisfying $(B1)$--$(B4)$, $\mu_{D}$-a.s.\ there
are no infinite present Type-II clusters or infinite absent
Type-II clusters; morevoer there are no infinite contours.
\end{theorem}

Without (B4), we have a weaker conclusion.

\begin{theorem}\label{m17}For given edge weights satisfying $(B1)$--$(B3)$, $\mu_{D}$-a.s.\ there is at most one infinite contour.
\end{theorem}

Next, we discuss the XOR Ising model. Consider an Ising model with spins
located on vertices of the dual square grid $\LL_2$. Assume each edge $e$ of
$\LL_2$ has coupling constant $J_e>0$. In order to make the connection to the
dimer model, note that $e$ crosses exactly one square face $S_e$ of the
square-octagon lattice; see Figure \ref{sol}. Let $e'$ be either of the two
sides of $S_e$ parallel to $e$. Assume the coupling constant $J_e$ and the
edge weight $w_{e'}$ satisfy the following identity
\begin{eqnarray}
w_{e'}=\frac{2\exp(-2J_e)}{1+\exp(-4J_e)}\label{we}
\end{eqnarray}
When $w_e\in(0,1)$, there is a unique $J_e>0$ satisfying identity (\ref{we}).

The XOR Ising model (see \cite{DBW11}) is a random spin
configuration on $\LL_2$ given by
\begin{eqnarray*}
\sigma_{XOR}(v)=\sigma_1(v)\sigma_2(v), \quad v\in V(\LL_2)
\end{eqnarray*}
where $\sigma_1$, $\sigma_2$ are two independent Ising models on vertices of
$\LL_2$, taking values in $\{\pm 1\}^{V(\LL_2)}$. Assume both $\sigma_1$ and
$\sigma_2$ have coupling constants given by (\ref{we}), and both $\sigma_1$
and $\sigma_2$ are sampled according to the law of the Ising model obtained
as the weak limit under periodic boundary conditions; see \cite{ZL12}. A
\df{``$+$''-cluster} (resp.\ \df{``$-$''-cluster}) of an XOR Ising
configuration, whose spins are located on vertices of $\LL_2$, is a maximal
connected set of vertices of $\LL_2$ in which every spin has state ``$+$''
(resp.\ ``$-$'') in $\sigma_{XOR}$. Similarly we can define an XOR Ising
model with spins located on vertices of $\LL_1$.

A \df{contour configuration} for an XOR Ising configuration, $\sigma_{XOR}$,
defined on $\LL_2$ (resp.\ $\LL_1$), is a subset of $\{0,1\}^{E(\LL_1)}$
(resp.\ $\{0,1\}^{E(\LL_2)}$), whose state-1-edges (present edges) are edges
of $\LL_1$ (resp.\ $\LL_2$) separating neighboring vertices of $\LL_2$ (resp.
$\LL_1$) with different states of $\sigma_{XOR}$. (Recall that $\LL_1$ and
$\LL_2$ are planar duals of each other.)  Contour configurations of the XOR
Ising model were first studied in \cite{DBW11}, in which the scaling limits
of contours of the critical XOR Ising model are conjectured to be level lines
of Gaussian free field. It is proved in \cite{bd14} that the contours of the
XOR Ising model on the square grid correspond to level lines of height
functions of the dimer model on the square-octagon lattice, inspired by the
correspondence between Ising model and bipartite dimer model in \cite{Dub}.
We will study  the percolation properties of the XOR Ising model, as an
application of the main theorems proved in this paper for the general
constrained percolation process.

Before stating the results on the percolation properties of the XOR Ising model, we identify the critical and non-critical phases for the family of XOR Ising models under consideration. Consider an Ising model, with spins located on vertices of $\LL_2$ and coupling constants obtained from dimer edge weights of the square-octagon lattice by (\ref{we}), such that the dimer edge weights satisfy Assumptions (B1)--(B3). Under the translation-invariance assumption (B1), the Ising model obtained above has the same coupling constant on all the horizontal edges (denoted by $J_h$), and the same coupling constant on all the vertical edges (denoted by $J_v$).

Let
\begin{eqnarray}
F(x,y):=\exp(-2x)+\exp(-2y)+\exp(-2x-2y).\label{f}
\end{eqnarray}
An Ising model on the square grid with coupling constants $J_h\geq 0$ on each horizontal edge and $J_v\geq 0$ on each vertical edge is said to be \df{critical} if
\begin{eqnarray}
F(J_h,J_v)=1.\label{ci}
\end{eqnarray}
The Ising model is said to be in the \df{low temperature state} if
\begin{eqnarray}
F(J_h,J_v)<1.\label{li}
\end{eqnarray}
The Ising model is said to be in the \df{high temperature state} if
\begin{eqnarray}
F(J_h,J_v)>1.\label{hii}
\end{eqnarray}
It is known that in the high temperature state, the Ising model has a unique Gibbs measure, and the spontaneous magnetization vanishes; while in the low temperature state, the Gibbs measures are not unique and the spontaneous magnetization is strictly positive under the ``$+$''-boundary condition. See  \cite{JL72, Ai80,ZL12}.

 We claim that if the dimer edge weights also satisfy (B4), then the Ising model has critical coupling constants; otherwise the Ising model has non-critical coupling constants.
See \cite{ZL12}. It is straightforward to check that given (\ref{we}) and (B1)--(B3), (\ref{ci}) is equivalent to (B4).

We define the \df{critical XOR Ising model} (resp.\ \df{non-critical XOR Ising model}) to be one obtained from the product of two independent critical Ising models on a square grid such that
\begin{enumerate}[label=\Roman*.]
\item each Ising model has coupling constants $J_h$ on horizontal edges, and $J_v$ on vertical edges, such that $J_h$, $J_v$ satisfy (resp.\ do not satisfy) (\ref{ci});
\item each Ising model has a probability measure that is the weak limit of measures on finite graphs with periodic boundary conditions.
\end{enumerate}

 \begin{theorem}\label{cxor}For the critical XOR-Ising model as defined above,
\begin{enumerate} [label=\Roman*.]
\item almost surely there are no infinite \text{``$+$''}-clusters, and no infinite \text{``$-$''}-clusters;
\item almost surely there are no infinite contours.
 \end{enumerate}
 \end{theorem}

  Now let us turn to the non-critical XOR Ising model.

 \begin{theorem}\label{uc}The non-critical XOR Ising model, as defined above, almost surely has at most one infinite contour.
\end{theorem}

An XOR Ising model $\sigma_{XOR}=\sigma_1\sigma_2$ is said to be in the \df{low temperature state} (resp.\ \df{high temperature state}) if both $\sigma_1$ and $\sigma_2$ are in the low temperature state (resp.\ high temperature state). Recall that both $\sigma_1$ and $\sigma_2$ have the same parameters $J_h$ and $J_v$.

\begin{theorem}\label{nicl}In the low temperature XOR Ising model, almost surely there are no infinite contours.
\end{theorem}

For the high-temperature XOR Ising model, we can prove the existence of a unique infinite contour in sufficiently high temperature as follows.

\begin{theorem}\label{ehel}Let $p_c$ be the critical probability for the i.i.d Bernoulli site percolation on the square grid. Note that $p_c>\frac{1}{2}$. Let $h_0>0$ satisfy
\begin{eqnarray*}
\frac{e^{h_0}}{e^{h_0}+e^{-h_0}}=p_c.\label{h0}
\end{eqnarray*}
Consider a high-temperature XOR Ising model on the square grid, in which each horizontal edge has coupling constant $J_h\geq 0$, and each vertical edge has coupling constant $J_v\geq 0$ satisfying (\ref{hii}). If $2(J_h+J_v)<h_0$, then almost surely
\begin{enumerate}[label=\Roman*.]
\item there are no infinite ``$+$''-clusters or infinite ``$-$''-clusters;
\item there exists exactly one infinite contour.
\end{enumerate}
Moreover, let $J_h'>0$, $J_v'>0$ be obtained from $J_h$, $J_v$ by
\begin{eqnarray}
\exp(-2J_h)+\exp(-2J_v')+\exp(-2J_h-2J_v')=1;\label{da1}\\
\exp(-2J_h)+\exp(-2J_v')+\exp(-2J_h-2J_v')=1.\label{da2}
\end{eqnarray}
If we assign the coupling constant $J_h'$ to each horizontal edge of the square grid, and $J_v'$ to each vertical edge, then we obtain a low-temperature XOR Ising model in which the total number of infinite ``$+$''-clusters and ``$-$''-clusters is exactly one almost surely.
\end{theorem}

The paper is organized as follows. In \Cref{cocl}, we prove combinatorial results regarding configurations of contours and clusters. In \Cref{pf1}, we prove \Cref{m1}.  In \Cref{pm4}, we prove \Cref{m4}.   In \Cref{pf23}, we prove \Cref{m2}. In \Cref{pf1w}, we prove a few combinatorial and probabilistic results in preparation to prove the remaining main theorems of the paper.  In \Cref{ucpc}, we discuss similar combinatorial results in  unconstrained percolation. In \Cref{s8},  we prove \Cref{m3}. In \Cref{pf4}, we prove \Cref{m5}. In \Cref{pap}, we prove \Cref{dc,m17,cxor,uc,nicl,ehel}. In \Cref{p51}, we prove a combinatorial lemma required for the proof of \Cref{m4}.

\section{Contours and Clusters}\label{cocl}

In this section, we prove some combinatorial and
probabilistic results regarding infinite contours and
infinite clusters in constrained percolation processes, in
preparation for the proofs of  \Cref{m1,m2,m3,m4}. It will
frequently be convenient to consider graphs embedded into
the plane in the usual way; we identify and edge with
closed line segment joining its the endpoints.

 We begin with the following elementary lemma.

\begin{lemma}\label{ss}Consider the four vertices of a white face of $G$.   If in a constrained configuration in $\Omega$, all the four vertices have state 0, then flipping the states of all vertices to 1, and preserving the states of all the other vertices of $G$, we obtain another constrained configuration in $\Omega$. Similarly, we may change the states of all the four vertices of one white face of $G$ from 1 to 0, and obtain another configuration in $\Omega$.
\end{lemma}

\begin{proof}It is easy to verify that each of the four adjacent black faces has configurations satisfying the constraint.
\end{proof}

We introduce an \df{augmented square grid} $A\ZZ^2$, whose vertices are either vertices of $G$, centers of faces of $G$, or midpoints of edges of $G$. Two vertices $u,w$ of $A\ZZ^2$ are joined by an edge of $A\ZZ^2$ if and only if $\|u-v\|_1=\frac{1}{2}$.

Note that  $A\ZZ^2$ is a square grid with edge length $\frac{1}{2}$. Let $[A\ZZ^2]^*$ be the planar dual lattice of $A\ZZ^2$, which is also a square grid with edge length $\frac{1}{2}$.

For any edge $e\in E(\LL_1)\cup E(\LL_2)$, consider the rectangle $R(e)$ in $\RR^2$ consisting of all the points within $\ell^{\infty}$ distance at most $\frac{1}{4}$ of the line segment joining its endpoints. For a contour $C$, let $\widetilde{C}=\{R(e):e\in C\}$. The topological boundary of $\widetilde{C}$ in $\RR^2$ is precisely a union of line segments corresponding to a set $S$ of edges of $[A\ZZ^2]^*$. The \df{interface} of $C$ is this set of edges $S$. See Figure \ref{pdi}.

In particular, each component of the interface of the contour $C$ is either a self-avoiding cycle or a doubly-infinite self-avoiding path, consisting of edges of $[A\ZZ^2]^*$, and each vertex of $[A\ZZ^2]^*$ is incident to 0 or 2 edges in the interface. Here by self-avoiding cycle we mean a finite connected component of edges of $[A\ZZ^2]^*$ in which each vertex of $[A\ZZ^2]^*$ has two incident edges.

The \df{interface of a contour configuration} $\phi\in\Phi$ is the union of interfaces of all contours in $\phi$. Each component of the interface of $\phi$ is a self-avoiding cycle or doubly-infinite self-avoiding path in $[A\ZZ^2]^*$. See Figure \ref{pdi} for an example of the interface. See also \cite{GH}.

\begin{lemma}\label{cin}For any contour configuration $\phi\in \Phi$, contours can never intersect interfaces, when interpreted as subsets of $\RR^2$.
\end{lemma}
\begin{proof}Any intersection must lie in the interior of either a black square or a white square. It is straightforward to check the two cases separately. In the case of a black square, we use the fact that primal contours and dual contours cannot cross each other.
\end{proof}

Throughout this section, we let $\omega\in \Omega$ be a constrained percolation configuration, and let $\phi$ be the corresponding contour configuration in $\LL_1\cup\LL_2$.

\begin{figure}
 \centering
  \includegraphics[width=0.6\textwidth]{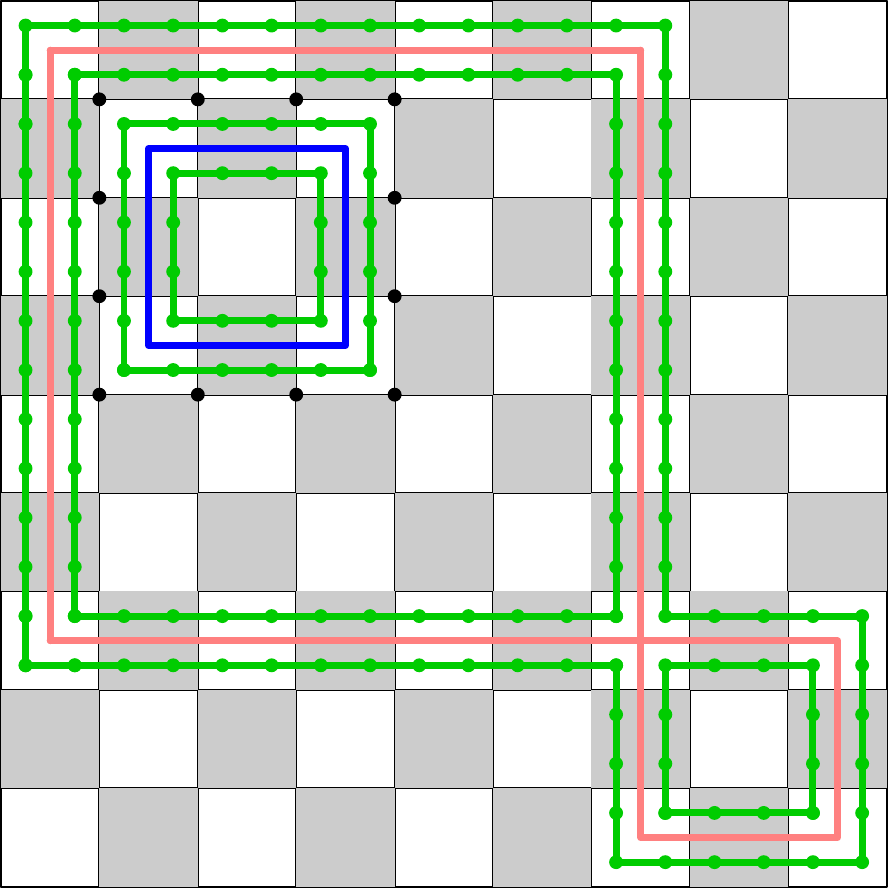}
 \caption{Primal contour, dual contour and interface: red lines represent primal contours; blue lines represent dual contours, green lines represent interfaces; green vertices represent vertices of $[A\ZZ^2]^*$ along the interface; black vertices represent part of a cluster.}\label{pdi}
 \end{figure}

 \begin{lemma}\label{rl}For any component $I$ of the interface of $\phi$, Let $F_{I}$ be the set consisting of all the vertices of $G$ whose $\ell^{\infty}$ distance to $I$ is $\frac{1}{4}$. Then all the vertices in $F_I$ lie in a single cluster, and $F_I$ is the vertex set of a doubly-infinite self-avoiding path (resp.\ self-avoiding cycle) if $I$ is a doubly-infinite self-avoiding path (resp.\ self-avoiding cycle).
 \end{lemma}
 \begin{proof}First of all, note that $F_I$ is a connected set of vertices in $G$. Now, if not all the vertices in $F_I$ are in the same infinite cluster, then there exist a pair of adjacent vertices $x,y\in F_I$, such that the edge $(x,y)$ of $G$ crosses a contour. Then the contour crossing $(x,y)$ must cross the interface $I$ as well, but this is a contradiction to \Cref{cin}.

Next we show that $F_I$ is the vertex set of a doubly infinite self-avoiding path or a self-avoiding cycle. Let $E_I$ be the set of edges of $G$ whose distance to $I$ is $\frac{1}{4}$. It is straightforward to check that $F_I=V(E_I)$. The fact that $E_I$ has degree 2 follows by local case analysis.

 See Figure \ref{pdi} for an example of such a part of a cluster, represented by black points.
 \end{proof}

 \begin{lemma}\label{l24}Let $\mathcal{C}$ be a nonempty collection of contours. Any two vertices $x,y\in \ZZ^2$ in a connected component of $\RR^2\setminus \mathcal{C}$ are connected by a path in $G$ that lies in the same component of $\RR^2\setminus \mathcal{C}$.
 \end{lemma}
 \begin{proof}For each $x\in \ZZ^2$, let $S_x=x+[-\frac{1}{2},\frac{1}{2}]^2$ be the unit square centered at $x$. Consider a path $p_{xy}$ in $\RR^2\setminus\mathcal{C}$ from $x$ to $y$. Let
 \begin{eqnarray*}
 A=\{x\in\ZZ^2:p_{xy}\ \text{intersects}\ S_x\}.
 \end{eqnarray*}
 Then we claim that $A$ is a connected set in $G\setminus \mathcal{C}$. Indeed, when the path enters a new square, it must do so either by crossing an edge of $G^*$ (the dual graph of $G$), or by passing through a vertex of $G^*$. In the former case, the edge of $G^*$ does not lie in a contour of $\mathcal{C}$; in the latter case, the vertex of $G^*$ does not lie in a contour of $\mathcal{C}$.
 \end{proof}

 In the following lemma, contours may be primal or dual as usual. We say a cluster is \df{incident} to a contour, if there exists a vertex of $\ZZ^2$ in the cluster and an edge of $\LL_1$ or $\LL_2$ in the contour, such that the Euclidean distance of the vertex and the contour is $\frac{1}{2}$.

\begin{lemma}\label{cc}Consider any configuration $\omega\in \Omega$ and associated contour configuration $\Phi=\Phi(\omega)$.
\begin{enumerate}[label=\Roman*.]
\item If there exist at least two infinite contours, then there exists an infinite 0-cluster or an infinite 1-cluster.
\item If $C_1$ and $C_2$ are two infinite contours, then there exists an infinite cluster incident to $C_1$.
\item If $\xi$ is a cluster incident to two infinite contours, then $\xi$ is an infinite cluster.
\item Suppose that $C_1$ is an infinite contour and $\mathcal{C}$ is a nonempty collection of infinite contours, such that $C_1\notin\mathcal{C}$. Let $R$ be the unbounded component of $\RR^2\setminus \cup_{C\in\mathcal{C}}C$ containing $C_1$. Then there is an infinite cluster in $R$.
\end{enumerate}
\end{lemma}
\begin{proof}We first prove II., which immediately implies I.. If there exist at least two infinite contours, then we can find two distinct infinite contours $C_1$ and $C_2$, two points $x\in C_1$ and $y\in C_2$ (midpoints of edges of $G$), and a self-avoiding path $p_{xy}$, consisting of edges of $G$ and two half-edges, one starting at $x$ and the other ending at $y$, and connecting $x$ and $y$, such that $p_{xy}$ does not intersect any infinite contours except at $x$ and at $y$. Indeed, we may take any path intersecting two distinct infinite contours, and then take a minimal subpath with this property.

Let $v\in \ZZ^2$ be the first vertex along $p_{xy}$ starting from $x$. Let $u$ be the midpoint of the line segment $[v,x]$. Then $u$ lies on the interface of $C_1$. Let $\ell_u$ be the connected component of the interface of $C_1$ containing $u$. Then $\ell_u$ is either a doubly-infinite self-avoiding path or a self-avoiding cycle  consisting of edges of $[A\ZZ^2]^*$.

We consider these two cases separately. Firstly, if $\ell_u$ is a doubly-infinite self-avoiding path, then we claim that $v$ is in an infinite (0 or 1-)cluster of the constrained site configuration on $\ZZ^2$. Indeed, by \Cref{rl}, all the vertices in $F_{\ell_u}$ are in the same cluster and $F_{\ell_u}$ is a doubly infinite self-avoiding path in $G$.

Secondly, if $\ell_u$ is a self-avoiding cycle, then considering $\ell_u$ as a union of line segments in $\RR^2$, $\RR^2\setminus \ell_u$ has two components, $Q_v$ and $Q_v'$, where $Q_v$ is the component including $v$. Since $\ell_u$ is a cycle, exactly one of $Q_v$ and $Q_v'$ is bounded, and the other is unbounded. Using \Cref{cin}, $x\in Q_v'$ implies $C_1\in Q_v'$. Since $C_1\subseteq Q_v'$, and $C_1$ is an infinite contour, we deduce that $Q_v'$ is unbounded, and $Q_v$ is bounded. Note that $y\notin \ell_u$ by \Cref{cin}, so either $y\in Q_v$, or $y\in Q_v'$. If $y\in Q_v'$, then any path consisting of edges of $G$ and one half-edge incident to $y$ and connecting $v$ and $y$ must intersect $\ell_u$, and therefore must intersect $C_1$ also. In particular, $p_{xy}$ intersects $C_1$ not only at $x$, but also at some point other than $x$. This contradicts the definition of $p_{xy}$. Hence $y\in Q_v$. By \Cref{cin}, this implies $C_2\subseteq Q_v$. But $C_2\subseteq Q_v$ is impossible since $C_2$ is infinite and $Q_v$ is bounded. Hence this second case is impossible.

Therefore we conclude that  there exists an infinite (0 or 1)-cluster incident to $C_1$. This establishes II., and hence I..

We now turn to III.. Assume that $\xi$ is a cluster incident to two distinct infinite contours $C_1$ and $C_2$. We can find a path $p_{xy}$, as above, such that every vertex of $G$ along $p_{xy}$ is in $\xi$. Then $\xi$ is infinite since the interface $\ell_u$ is infinite. This establishes III..

Consider Part IV.\ of the lemma. We say a contour is \df{incident} to $R$, if there exists an edge $e$ of $\LL_1$ or $\LL_2$ in the contour and a vertex $v$ of $G$ in $R$, such that the Euclidean distance of $e$ and $v$ is $\frac{1}{2}$. We claim that there exists at least one infinite contour in $\mathcal{C}$ incident to $R$. Recall that $R$ is a connected component of $\RR^2\setminus \mathcal{C}$. Indeed, if there is no infinite contour in $\mathcal{C}$ incident to $R$, then $R=\RR^2$. But this is impossible since $\mathcal{C}$ is nonempty.

 Let $C_2\in \mathcal{C}$ be an infinite contour in $\mathcal{C}$ incident to $R$. Since $R$ contains at least one infinite contour, we can find an infinite contour $C_3\subset R$ such that there exists a path $\ell_{xy}$ connecting a point $x\in C_2$ and $y\in C_3$, consisting of a half-edge starting from $x$, a half-edge starting from $y$ and edges of $G$, such that $\ell_{xy}$ crosses no infinite contours except at $x$ and at $y$, and all the vertices of $G$ along $\ell_{xy}$ are in $R$ by \Cref{l24}. Following the same procedure as above we can find an infinite cluster in $R$ adjacent to $C_2$ and $C_3$.
\end{proof}

\Cref{cc} has the following straightforward corollary.

\begin{corollary}\label{ncl} Let $\mu$ be a probability measure on $\Omega$ satisfying Assumptions $(A1)$--$(A3)$. If $\mu$-a.s.\ there are no infinite clusters, then $\mu$-a.s.\ there are no infinite contours.
\end{corollary}

\begin{proof}Let $\sA_1$ (resp.\ $\sA_2$) be the event that there exists at least one infinite primal (resp.\ dual) contour. There is a bijection between configurations in $\sA_1$ and configurations in $\sA_2$; specifically, we translate each configuration in $\sA_1$ by $(1,1)$, and obtain a configuration in $\sA_2$, and vice versa.  By Assumption (A3), we have $\mu(\sA_1)=\mu(\sA_2)$.

Moreover, since $\sA_1$ and $\sA_2$ are $2\ZZ\times 2\ZZ$ translation invariant events, by Assumption (A2), we have either $\mu(\sA_1)=\mu(\sA_2)=0$, or $\mu(\sA_1)=\mu(\sA_2)=1$.

Suppose that  $\mu(\sA_1)=\mu(\sA_2)=1$.  Since primal contours and dual
contours are distinct, $\mu$-a.s.\ there exist at least two distinct infinite
contours. By \Cref{cc} I., $\mu$-a.s.\ there exists an infinite cluster.
\end{proof}

If $C$ is a contour, we write $G\setminus C$ for the subgraph obtained from $G$ by removing all the edges of $G$ crossed by edges of $C$.

\begin{lemma}\label{icic}Let $C_{\infty}$ be an infinite contour. Then each infinite component of $G\setminus C_{\infty}$ contains an infinite cluster that is incident to $C_{\infty}$.
\end{lemma}
\begin{proof} Let $S$ be an infinite component of $G\setminus C_{\infty}$. Let $x\in C_{\infty}$ be the midpoint of an edge of $G$, and let $y\in S$ be a vertex of $G$, such that the Euclidean distance of $x$ and $y$ is $\frac{1}{2}$. Let $v$
be the midpoint of the line segment $[x,y]$. Then $v$ lies on the interface of $C_{\infty}$. Let $\ell_v$ be the component of the interface of $C_{\infty}$ containing $v$.

We claim that $\ell_v$ is infinite. Suppose that $\ell_v$ is finite. Then $\ell_v$ is a self-avoiding cycle. Let $Q_x$ (resp.\ $Q_y$) be the component of $\RR^2\setminus \ell_v$ containing $x$ (resp.\ $y$). Then exactly one of $Q_x$ and $Q_y$ is bounded, and the other is unbounded. Note that $C_{\infty}\subset Q_x$ by \Cref{cin}.

We claim that $S\subset Q_y$. To see why that is true, note that since $S$ is connected and $y\in S\cap Q_y$, if $S$ is not a subset of $Q_y$, there exist a pair of adjacent vertices $p,q\in S$, such that $p\in Q_y$ and $q\notin Q_y$. Then the edge $\langle x,y \rangle$ of $G$ crosses the interface $\ell_v$, and therefore crosses the contour $C_{\infty}$ as well. But this is impossible since $S$ is an infinite component of $G\setminus C_{\infty}$.

 Since it is impossible that $C_{\infty}\in Q_x$ and $S\subset Q_y$ both $C_{\infty}$ and $S$ are infinite, we infer that $\ell_v$ is infinite.

According to \Cref{rl}, all the vertices in $F_{\ell_v}$ lie in an infinite cluster incident to $C_{\infty}$.
\end{proof}

\begin{lemma}\label{io} Let $x\in\ZZ^2$ be in an infinite 0-cluster, let $y\in\ZZ^2$ be in an infinite 1-cluster, and  let $\ell_{xy}$ be a path, consisting of edges of $G$ and connecting $x$ and $y$. Then $\ell_{xy}$ has an odd number of crossings with infinite contours in total.

In particular, if there exist both an infinite 0-cluster and an infinite 1-cluster, then there exists an infinite contour.

\end{lemma}

\begin{proof}Moving along $\ell_{xy}$, two neighboring vertices $u,v\in \ZZ^2$ of $\ell_{xy}$ have different states if and only if the edge $\langle u,v\rangle$ crosses a contour. Since the states of $x$ and $y$ are different, moving along $\ell_{xy}$, the states of vertices must change an odd number of times.  Therefore $\ell_{xy}$ crosses (primal and dual) contours an odd number of times.

It remains to show that the total number of crossings of $\ell_{xy}$ with finite contours is even. Since $\ell_{xy}$ crosses finitely many finite contours in total, let $C_{1},\ldots,C_{m}$ be all the finite contours intersecting $\ell_{xy}$, where $m$ is a nonnegative integer.

Let $G\setminus \cup_{i=1}^{m}C_{i}$ be the subgraph obtained from $G$ by removing all the edges of $G$ crossed by the $C_i$'s.
Since all the $C_i$'s are finite, $G\setminus \cup_{i=1}^{m}C_{i}$ has exactly one infinite component. We claim that both $x$ and $y$ lie in the infinite connected component of $G\setminus \cup_{i=1}^{m}C_{i}$. Indeed, if $x$ is in a finite component of $G\setminus \cup_{i=1}^{m}C_{i}$, then it is a contradiction to the fact that $x$ is in an infinite 0-cluster, because the infinite 0-cluster including $x$ cannot be a subset of a finite component of $G\setminus \cup_{i=1}^{m}C_{i}$. Similarly $y$ is also in an infinite component of $G\setminus \cup_{i=1}^{m}C_{i}$. Since $G\setminus \cup_{i=1}^{m}C_{i}$ has a unique infinite component, we infer that both $x$ and $y$ are in the same infinite component of $G\setminus \cup_{i=1}^{m}C_{i}$.

Since both $x$ and $y$ lie in the infinite connected component of $G\setminus \cup_{i=1}^{m}C_{i}$, we can find a path $\ell'_{xy}$ connecting $x$ and $y$, using edges of $G$, such that the path does not intersect $\cup_{i=1}^{m}C_{i}$ at all. Moreover, each vertex of $\LL_1$ or $\LL_2$ has an even number of incident edges in $\cup_{i=1}^{m}C_{i}$. We can transform $\ell_{xy}$ to $\ell'_{xy}$ using a finite sequence of moves; in each move, the path only changes along the boundary of a single face of $G$. Since the face contains either no vertex of $V(\LL_1)\cup V(\LL_2)$, or a single vertex of even degree in $\cup_{i}C_i$, it is easy to verify that the parity of the total number of crossings is preserved. This implies that $\ell_{xy}$ must cross infinite contours an odd number of times, because $\ell_{xy}$ crosses (infinite and finite) contours an odd number of times in total, and $\ell_{xy}$ crosses finite contours an even number of times.
\end{proof}

\begin{lemma}\label{isc}Assume that $\xi$ is an infinite cluster, and $C$ is an infinite contour. Assume that $x$ is a vertex of $G$ in  $\xi$, and let $y\in C$ be the midpoint of an edge of $G$. Assume that there exists a path $p_{xy}$ connecting $x$ and $y$, consisting of edges of $G$ and a half-edge incident to $y$, such that $p_{xy}$ crosses no infinite contours except at $y$.  Let $z$ be the first vertex of $\ZZ^2$ along $p_{xy}$ starting from $y$. Then $z\in\xi$.
\end{lemma}

\begin{proof}Since $p_{xy}$ crosses no infinite contours except at $y$, let $C_{1},\ldots,C_{m}$ be all the finite contours crossing $p_{xy}$.  We claim that $\RR^2\setminus \cup_{i=1}^{m}C_{i}$ has a unique unbounded component, which contains both $x$ and $y$. Indeed,  since $x\in \xi$ and $y\in C$; neither the infinite cluster $\xi$ nor the infinite contour $C$ can lie in a bounded component of $\RR^2\setminus \cup_{i=1}^{m}C_{i}$.

Let $I$ be the intersection of the interface of $\cup_{i=1}^{m}C_{i}$ with the unique unbounded component of $\RR^2\setminus \cup_{i=1}^{m}C_{i}$. Since each $C_{i}$, $1\leq i\leq m$, is a finite contour, each component of the interface of $C_{i}$ is finite. In particular, $I$ consists of finitely many disjoint self-avoiding cycles, denoted by $D_1,\ldots,D_t$. For $1\leq i\leq t$, $\RR^2\setminus D_i$ has exactly one unbounded component, and one bounded component.

By \Cref{rl}, each $F_{D_i}$ form a self-avoiding cycle of $G$, and all the vertices in $F_{D_i}$, for each fixed $i$, are in the same cluster. Note that each time $p_{xy}$ crosses $D_i$, it must intersect $F_{D_i}$ at a vertex of $G$. We claim that all the vertices in $\cup_{i=1}^{t}F_{D_i}$, are also in the same cluster. Indeed, $p_{xy}$ is divided by crossings with the interfaces $D_i$ ($1\leq i\leq t$) into nonoverlapping segments; the interior of each segment is either in a bounded component of $\RR^2\setminus \cup_{i=1}^{t}D_i$, or in the unbounded component of $\RR^2\setminus\cup_{i=1}^{t} D_i$. See Figure \ref{cyc}. 
\begin{figure}
\includegraphics{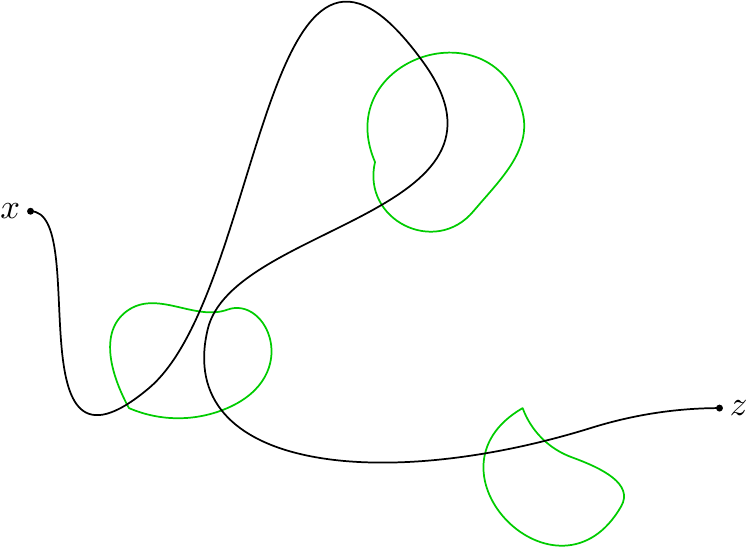}
\caption{$p_{xy}$ and finite interfaces crossing $p_{xy}$: green cycles represent interfaces, the black line represent $p_{xy}$. }\label{cyc}
\end{figure}
For each segment of $p_{xy}$ whose interior lies in an unbounded component of $\RR^2\setminus \cup_{i=1}^{t}D_i$, all the vertices of $G$ on the segment are in the same cluster (since the segment crosses no contours), including two vertices in $F_{D_i}$ and $F_{D_j}$, for some $i\neq j$. Since for any $i,j$, $i\neq j$, $F_{D_i}$ and $F_{D_j}$ can be connected by finitely many such steps as described above, we conclude that all the vertices in $\cup_{i=1}^{t}F_{D_i}$ are in the same infinite cluster. Similarly, $x$ is in the same cluster as $F_{D_i}$, for some $i$, and $z$ is in the same cluster as $F_{D_j}$, for some $j$. Therefore we have $z\in\xi$.
\end{proof}

Using the same arguments as in \Cref{isc}, we can prove the following.

\begin{lemma}\label{cscc}Assume that $C_1$, $C_2$ are two infinite contours. Assume that $x$, $y$ be the midpoints of edges of $G$. Assume that there exists a path $p_{xy}$ connecting $x$ and $y$, consisting of edges of $G$ and two half-edges incident to $x$ and $y$, such that $p_{xy}$ crosses no infinite contours except at $x$ and $y$.  Let $z$ (resp.\ $w$) be the first (resp.\ last) vertex of $\ZZ^2$ along $p_{xy}$ starting from $y$. Then $z, w$ are in the same cluster.
\end{lemma}

\section{Non-existence of finitely many infinite clusters}\label{pf1}

In this section, we prove Theorem \ref{m1}. Throughout this section, let $\omega\in\Omega$ and let $\phi$ be the associated contour configuration.

The number of \df{ends} of a connected graph is the supremum over its finite subgraphs of the number of infinite components that remain after removing the subgraph.

\begin{lemma}\label{att}Let $\mu$ be a probability measure on $\Omega$ satisfying (Ak1). Then $\mu$-a.s.\ no contour has more than two ends.
\end{lemma}
\begin{proof}
The lemma follows from remark after Corollary 5.5 of
\cite{blps}; see also Exercise 7.24 of \cite{lp}
or Lemma 4.5 of \cite{HGK14}.
\end{proof}

\begin{lemma}\label{ec2}Let $C$ be an infinite contour with at most 2 ends. Then $\RR^2\setminus C$ has at most 2 unbounded components.
\end{lemma}
\begin{proof}Assume that $\RR^2\setminus C$ has three unbounded components $\xi_1$, $\xi_2$ and $\xi_3$; we will obtain a contradiction.

We can find three points $u\in\xi_1$, $v\in\xi_2$ and $w\in\xi_3$ such that
there exist three semi-infinite paths $\ell_u\subset \xi_1$,
$\ell_v\subset\xi_2$ and $\ell_w\subset \xi_3$ starting from $u$, $v$ and
$w$, respectively. Moreover, there exists a simply connected domain $B\subset
\RR^2$ containing $u,v,w$.  We can choose the domain $B$ such that
$\RR^2\setminus [\ell_u\cup\ell_v\cup\ell_w\cup B]$ has exactly 3 unbounded
components, denoted by $\eta_1$, $\eta_2$ and $\eta_3$, such that $\eta_1$
(resp.\ $\eta_2$, $\eta_3$) is incident to $\ell_v$ and $\ell_w$ (resp.
$\ell_u$ and $\ell_w$, $\ell_u$ and $\ell_v$). Since $C\cap
[\ell_u\cup\ell_v\cup\ell_w]=\emptyset$, and $C\setminus B$ has at most two
infinite components (this follows from the fact that $C$ has at most two
ends), at least one of $\eta_1$, $\eta_2$, $\eta_3$ does not include an
infinite component of $C\setminus B$. Without loss of generality, assume that
$\eta_1$ does not include an infinite component of $C\setminus B$. Then we
can find a path $\ell_{vw}$ connecting $\ell_v$ and $\ell_w$, such that
$\ell_{vw}\cap C=\emptyset$. Hence $\xi_2$ and $\xi_3$ are the same component
of $\RR^2\setminus C$. But this is impossible.

Therefore we conclude that if $C$ is an infinite contour with 2 ends, $\RR^2\setminus C$ has at most 2 unbounded components.
\end{proof}

Let $\mathcal{C}_1$ (resp.\ $\mathcal{C}_2$, $\mathcal{C}_0$) be the set of infinite contours $C$ such that $\RR^2\setminus C$ has exactly 1 (resp.\ 2,0) unbounded components.

\begin{lemma}\label{lc}If $\mathcal{C}_1\cup \mathcal C_2\neq \emptyset$, then $\mathcal{C}_0=\emptyset$.
\end{lemma}
\begin{proof}If $\mathcal{C}_0\neq \emptyset$, i.e.\ $\mathcal{C}_0$ contains at least one contour $C$, then $C$ is the only infinite contour, since every other contour lies in a component of $\RR^2\setminus C$, and $\RR^2\setminus C$ has only finite components.  But this is impossible since $\mathcal{C}_1\cup\mathcal C_2\neq \emptyset$.
\end{proof}

\begin{lemma}\label{l33}Assume that every infinite contour in $\phi$ has at most 2 ends.  Assume that $\mathcal{C}_1\cup\mathcal{C}_2\neq \emptyset$. Let $m_2$ (resp.\ $m$) be the number of unbounded components in $\RR^2\setminus \cup_{C\in\mathcal{C}_2}C$ (resp.\ $\RR^2\setminus \cup_{C\in \mathcal{C}_1\cup \mathcal{C}_2}C$). Then $m=m_2$ (where possibly both are $\infty$.)
\end{lemma}

To prove \Cref{l33}, we first need a fact about metric spaces.

If $x$ and $y$ are two points of a locally connected metric space $M$, and $H\subseteq M$ is a closed set. We say $H$ \df{separates} $x$ from $y$, if $x$ and $y$ are in two distinct components of $M\setminus H$.

\begin{proposition}\label{p34}Let $x$ and $y$ be two points of a connected and locally arcwise connected metric space $S$, and let $G=\{g_i\}$ be a countable collection of closed sets such that
\begin{enumerate}
\item the common part of every pair of elements of $G$ is the closed set $H$ (which may be empty);
\item if $b_1$ and $b_2$ are two arcs from $x$ to $y$ that lie in $S\setminus H$, then $b_1\cup b_2$ lies in a compact set which is simply connected in the weak sense and whose closure contains no point of $H$;
\item $\cup_{g_i\in G}g_i$ is locally compact.
\end{enumerate}
If no element of $G$ separates $x$ from $y$ in $S$, then $\cup_{g_i\in G}g_i$ does not separate $x$ from $y$ in $S$.
\end{proposition}
\begin{proof}See Theorem 3 of \cite{Bas35}.
\end{proof}

\medskip

\noindent\textbf{Proof of \Cref{l33}.} Since every infinite
contour in $\phi$ has at most 2 ends, by \Cref{ec2}, the
complement of every infinite contour in $\phi$ has at most
2 unbounded components. In other words
$\mathcal{C}_1\cup\mathcal{C}_2\cup\mathcal{C}_0$ contains
all the infinite contours in $\phi$. Since
$\mathcal{C}_1\cup \mathcal{C}_2\neq \emptyset$, by
\Cref{lc}, $\mathcal{C}_0=\emptyset$. Therefore
$\mathcal{C}_1\cup\mathcal{C}_2$ contains all the infinite
contours in $\phi$.

 Let $\mathcal{Q}_2$ (resp.\ $\mathcal{Q}$) be the set of unbounded components in $\RR^2\setminus \cup_{C\in\mathcal{C}_2}C$ (resp.\ $\RR^2\setminus \cup_{C\in\mathcal{C}_2\cup \mathcal{C}_1}C$). To prove the lemma, it suffices to construct a bijection from $\mathcal{Q}_2$ to $\mathcal{Q}$.

For each $C\in\mathcal{C}_1$, define $f(C)$ to be the unbounded component of $\RR^2\setminus C$. For each $R\in \mathcal{Q}_2$, define
\begin{eqnarray*}
g(R)=R\cap \cap_{C\in \mathcal{C}_1}f(C).
\end{eqnarray*}
We claim that $g$ is a bijection from $Q_2$ to $Q$. To see why that is true, note that for each $C\in\mathcal{C}_1$, $C$ lies in an unbounded component of $\RR^2\setminus\cup_{C\in \mathcal{C}_2} C$. For each $R\in \mathcal{Q}_2$, the following cases might occur:
\begin{enumerate}[label=(\alph*)]
\item $R$ contains no infinite contours in $\mathcal{C}_1$;
\item $R$ contains at least one infinite contour in $\mathcal{C}_1$.
\end{enumerate}

If Case (a) occurs, then for each $C\in\mathcal{C}_1$, $R\subset f(C)$. Hence $g(R)=R\in \mathcal{Q}$.

If Case (b) occurs, by Part IV of \Cref{cc}, we can find an infinite cluster in $R$. Note that this infinite cluster must lie in $R\cap\cap_{C\in \mathcal{C}_1}f(C)$. Therefore $g(R)$ is nonempty and unbounded. According to \Cref{p34}, any two points in $g(R)$ cannot be separated by contours in $\phi$. As a result, $g(R)$ is connected, and there exists a unique component $Q\in\mathcal{Q}$ such that $g(R)\subseteq Q$.

We claim that $g(R)=Q\in \mathcal{Q}$. Indeed, if $g(R)$ is a proper subset of $Q$, then there exist $x\in g(R)$, $y\in[Q\setminus g(R)]$, such that $\cup_{C\in\mathcal{C}_1\cup\mathcal{C}_2}C$ separates $x$ from $y$ in $\RR^2$. But this is impossible since $x,y\in Q$.

Note that $g(R)\subseteq R$ for any $R\in\mathcal{Q}_2$. We claim that for any $R\in\mathcal{Q}_2$, $g(R)$ is the unique element in $\mathcal{Q}$ satisfying $g(R)\subseteq R$. Indeed, if there exists $Q,Q'\in\mathcal{Q}$, such that $Q\subseteq R$ and $Q'\subseteq R$, then using the definition of $g(R)$, we deduce that $Q\subseteq g(R)$ and $Q'\subseteq g(R)$. Since $g(R)$ is a component in $\mathcal{Q}$, and different components are disjoint, we have $g(R)=Q=Q'$.

Moreover, for each $Q\in\mathcal{Q}$, we can find a unique $R\in\mathcal{Q}_2$, such that $Q\subseteq R$. Then $Q=g(R)$. Hence $g$ is a bijection from $\mathcal{Q}_2$ to $\mathcal{Q}$; the proof is complete.
\hfill$\Box$

\begin{lemma}\label{eo}Assume that every infinite contour in $\phi$ has at most 2 ends.
 Let $\sN$ be the total number of infinite contours satisfying the condition that the complement of the infinite contour in $\RR^2$ has two unbounded components; i.e.\ $\sN=|\mathcal{C}_2|$.
 \begin{enumerate}
\item If $1\leq \sN< \infty$, then the total number of unbounded components in $\RR^2\setminus [\cup_{C\in\mathcal{C}_2}C]$ is $\sN+1$. Moreover, there is an infinite component of interface (doubly-infinite self-avoiding path) of some contour in $\mathcal{C}_2$ in each unbounded component of $\RR^2\setminus [\cup_{C\in\mathcal{C}_2}C]$.
\item If $\sN=\infty$, then there are infinitely many infinite clusters.
\end{enumerate}
\end{lemma}
\begin{proof}
We first prove Part I.
We will prove this statement by induction on $\sN$.

First of all, when $\sN=1$, let $C_1$ be the unique infinite contour in
$\mathcal{C}_2$. By definition of $\mathcal{C}_2$, obviously $\RR^2\setminus
C_1$ has 2 unbounded components, denoted by $\eta_1$ and $\eta_2$. Let
$x\in\eta_1\cap \ZZ^2$ (resp.\ $y\in \eta_2\cap \ZZ^2$) be a vertex of $G$ in
$\eta_1$ (resp.\ $\eta_2$) whose Euclidean distance to $C_1$ is
$\frac{1}{2}$. Let $p,q\in C_1$ be two points along $C_1$ which are midpoints
of edges of $G$ such that the Euclidean distance of $x$ and $p$ (resp.\ $y$
and $q$) is $\frac{1}{2}$. Let $I_1$ (resp.\ $I_2$) be the component of
interface of $C_1$ passing through the midpoint of $x$ and $p$ (resp.\ $y$
and $q$). Then $I_1$ (resp.\ $I_2$) is either a self-avoiding cycle or
doubly-infinite self-avoiding path.

We claim that both $I_1$ and $I_2$ are doubly-infinite self-avoiding paths. To see why that is true, it suffices to show that both components of $\RR^2\setminus I_1$ (resp.\ $\RR^2\setminus I_2$) are unbounded. Indeed, $I_1$ passes through the midpoint of $x$ and $p$, $x\in\eta_1\cap \ZZ^2$ and $p\in C_1$. By \Cref{cin}, the contour $C_1$ and the interface $I_1$ never intersect. Therefore $C_1$ lies in one component of $\RR^2\setminus I_1$, which is unbounded since $C_1$ is infinite. All the vertices in $\eta_1\cap \ZZ^2$ lies in the other component of $\RR^2\setminus I_1$. Otherwise since $\eta_1$ is connected, there exist a pair of adjacent vertices $u,v\in \eta_1\setminus I_1$, such that $u$ and $v$ are in different components of $\RR^2\setminus I_1$. Then the edge $\langle u,v\rangle $ of $G$ must cross $I_1$ and hence the contour $C_1$. But this is impossible since $\eta_1$ is an unbounded component of $\RR^2\setminus C_1$. We deduce that the other component of $\RR^2\setminus I_1$ containing $C_1$ is also unbounded, and hence $I_1$ is a doubly infinite self-avoiding path. Similarly we can show that $I_2$ is also a doubly-infinite self-avoiding path.

 Therefore the statement holds when $\sN=1$.

Let $k\geq 1$ be a positive integer. We assume that
\begin{itemize}
\item I holds when $\sN=k$.
\end{itemize}

Now assume that $\sN=k+1$. Let $C_1,\ldots, C_{k+1}$ be all the infinite contours in $\mathcal{C}_2$. By the induction hypothesis,  $\RR^2\setminus [\cup_{i=1}^k C_k]$ has $k+1$ unbounded components, denoted by $\eta_1,\ldots,\eta_{k+1}$.

Without loss of generality, assume that $C_{k+1}\subset \eta_{k+1}$. Let $\xi_1$ and $\xi_2$ be the two unbounded components of $\RR^2\setminus C_{k+1}$. Then we claim that $\eta_1,\ldots, \eta_k$, $\eta_{k+1}\cap \xi_1$, $\eta_{k+1}\cap\xi_2$ are all the unbounded components in $\RR^2\setminus \cup_{i=1}^{k+1}C_i$. To see why that is true, note first that for $1\leq i\leq k$, $\eta_i$ is an unbounded component of $\RR^2\setminus\cup_{i=1}^{k+1}C_i$.

For $\eta_{k+1}\cap \xi_1$ the following cases might occur
\begin{enumerate}[label=(\alph*)]
\item $\xi_1$ contains none of $C_1,\ldots, C_k$.
\item $\xi_1$ contains at least one of $C_1,\ldots, C_k$.
\end{enumerate}

In Case (a), we claim that $\eta_{k+1}\cap \xi_1=\xi_1$, which is an unbounded component of $\RR^2\setminus \cup_{i=1}^{k+1}C_i$. To see why that is true, we first show that $\xi_1\subset \eta_{k+1}$. Indeed, since $\xi_1$ is an unbounded component of $\RR^2\setminus C_{k+1}$, and $C_{k+1}\in \eta_{k+1}$, there exist a vertex $v\in \xi_1\cap\eta_{k+1}\cap \ZZ^2$ where the Euclidean distance of $v$ to $C_{k+1}$ is $\frac{1}{2}$. Hence $\xi_1\cap\eta_{k+1}\neq \emptyset$. Since $\xi_1$ is an unbounded component of $\RR^2\setminus C_{k+1}$, if $\xi_1$ is not a subset of $\eta_{k+1}$, there exist a pair of adjacent vertices $u,w\in\xi_1\cap\ZZ^2$, such that $u\in \eta_{k+1}$ and $v\notin\eta_{k+1}$.
Then the edge $\langle u,v \rangle$ of $G$ must cross one of the contours $C_1,\ldots C_k$. But this is impossible since $\xi_1$ contains none of $C_1,\ldots C_k$.

In Case (b), by \Cref{p34}, $\eta_{k+1}\cap\xi_1$ is a connected component of $\RR^2\setminus \cup_{i=1}^{k+1}C_i$. We claim that $\eta_{k+1}\cap \xi_1$ is unbounded. By \Cref{cc} III, it suffices to show that $\eta_{k+1}\cap \xi_1$ contains a cluster incident to two infinite contours. Since $\xi_1$ is an unbounded component of $\RR^2\setminus C_{k+1}$, $C_{k+1}\in \eta_{k+1}$, and $\xi_1$ contains at least one of $C_1,\ldots, C_k$, we can find a vertex $p\in\xi_1\cap \eta_{k+1}\cap\ZZ^2$ such that $p$ is adjacent to $C_{k+1}$ and a vertex $q\in \xi_1\cap \cap\ZZ^2$ such that $q$ is adjacent to one of $C_1,\ldots C_k$, and a path $\ell_{pq}$ connecting $p$ and $q$ and consisting of edges of $G$, such that $\ell_{pq}$ crosses no contours in $C_1,\ldots,C_{k+1}$. All the vertices along $\ell_{pq}$ are in a cluster incident to two distinct infinite contours, and are contained in $\eta_{k+1}\cap\xi_1$. Hence $\eta_{k+1}\cap \xi_1$ is an unbounded component of $\RR^2\setminus \cup_{i=1}^{k+1}C_i$.

Similarly, we can show that $\eta_{k+1}\cap\xi_2$ is an unbounded component of $\RR^2\setminus \cup_{i=1}^{k+1}C_i$. Moreover, $\RR^2\setminus \cup_{i=1}^{k+1}C_i$ has no unbounded components other than $\eta_1,\ldots, \eta_k$, $\eta_{k+1}\cap \xi_1$, $\eta_{k+1}\cap\xi_2$. Therefore, when $\sN=k+1$, $\RR^2\setminus \cup_{i=1}^{k+1}C_i$ has exactly $k+2$ unbounded components.

Note that by the induction hypothesis, there is an infinite component of
interface of some contour in $C_1,\ldots,C_k$ (resp.\ $C_{k+1}$) in each one
of $\eta_1,\ldots,\eta_{k+1}$ (resp.\ $\xi_1$, $\xi_2$). We will show that
there is an infinite component of interface of some contour in $C_1,\ldots,
C_{k+1}$ in each one of $\eta_{k+1}\cap\xi_1$ and $\eta_{k+2}\cap\xi_2$ as
well. Let $I$ be an infinite component of interface of some contour in
$C_1,\ldots,C_k$ in $\eta_{k+1}$. Then $I$ lies in either $\eta_{k+1}\cap
\xi_1$ or $\eta_{k+1}\cap\xi_2$. Without loss of generality, assume that $I
\subset\eta_{k+1}\cap \xi_1$. Let $I'$ be an infinite component of the
interface of $C_{k+1}$ in $\xi_2$. Since $C_{k+1}\subset\eta_{k+1}$, we have
$I'\subset \eta_{k+1}\cap \xi_2$. Note that for two collection of contours
$\mathcal{C}\subseteq \mathcal{C}'$, each component of interface of contours
in $\mathcal{C}$ is also a component of interface of contours in
$\mathcal{C}'$.

Now we prove Part II of the lemma. It suffices to show that for each positive integer $k\geq 1$, if $\sN\geq k$, then there exist at least $k$ infinite clusters.

Assume that $\sN\geq k$. Let $C_1,\ldots, C_k\in\mathcal{C}_2$ be $k$ infinite contours such that $\RR^2\setminus C_i$ has two unbounded components, for $1\leq k\leq i$. By Part I of the lemma, $\RR^2\setminus\cup_{i=1}^k C_i$ has $k+1$ unbounded components, and each unbounded component contains an infinite component of interface of some contour in $\mathcal{C}_2$. By \Cref{rl}, there is an infinite cluster in each unbounded component of $\RR^2\setminus\cup_{i=1}^k C_i$. Therefore the total number of infinite clusters is at least $k+1$. Then the proof is complete.
\end{proof}

\begin{lemma}\label{uic}Assume that every infinite contour in $\phi$ has at most two ends. If $\mathcal{C}_2=\emptyset$ and $\mathcal{C}_1\neq \emptyset$, then $\RR^2\setminus[\cup_{\mathcal{C}\in \mathcal{C}_1}C]$ has exactly one unbounded component. Moreover, the unbounded component of $\RR^2\setminus [\cup_{C\in\mathcal{C}_1\cup \mathcal{C}_2}C]$ contains an infinite cluster.
\end{lemma}
\begin{proof}Recall that for each $C\in\mathcal{C}_1$, $f(C)$ is the unbounded component of $\RR^2\setminus C$.

If $\mathcal{C}_1\neq \emptyset$, then any infinite cluster must lie in $\cap_{C\in\mathcal{C}_1}f(C)$. We claim that $\cap_{C\in\mathcal{C}_1}f(C)$ is unbounded and contains an infinite cluster. Indeed, for each $C\in\mathcal{C}_1$, we can find a vertex $v$ of $\ZZ^2$ in $f(C)$ whose Euclidean distance to $C$ is $\frac{1}{2}$. Then $v\in \cap_{C\in\mathcal{C}_1}f(C)$. The lemma obviously holds when $|\mathcal{C}_1|=1$. If $|\mathcal{C}_1|\geq 2$, then $v$ is in a cluster incident to at least two infinite contours. By \Cref{cc} III, $v$ is in an infinite cluster. Hence $\cap_{C\in\mathcal{C}_1}f(C)$ is unbounded and contains an infinite cluster. By \Cref{p34}, $\cap_{C\in\mathcal{C}_1}f(C)$ is a connected component of $\RR^2\setminus [\cup_{C\in\mathcal{C}_1\cup \mathcal{C}_2}C]$.
\end{proof}

\begin{lemma}\label{bf}Assume that every infinite contour in $\phi$ has at most two ends. Assume that $\mathcal{C}_1\cup\mathcal{C}_2\neq \emptyset$, and $|\mathcal{C}_2|<\infty$. Then every unbounded component of $\RR^2\setminus [\cup_{C\in\mathcal{C}_1\cup \mathcal{C}_2}C]$ contains exactly one infinite cluster.
\end{lemma}
\begin{proof}By \Cref{rl,l33,eo} (for the case $\mathcal{C}_2\neq\emptyset$) and \Cref{uic} (for the case $\mathcal{C}_2=\emptyset$) every unbounded component of $\RR^2\setminus[\cup_{C\in\mathcal{C}_1\cup \mathcal{C}_2}C]$ contains at least one infinite cluster.

Let $\eta$ be an unbounded component of $\RR^2\setminus[\cup_{C\in\mathcal{C}_1\cup \mathcal{C}_2}C]$. By \Cref{lc}, $\eta$ contains no infinite contours. .

 Assume that $\eta$ contains two infinite clusters $\xi_1$ and $\xi_2$; we will obtain a contradiction. Let $x\in\xi_1\cap\ZZ^2$ and $y\in\xi_2\cap\ZZ^2$ be two vertices in $\xi_1$ and $\xi_2$, respectively. We can find a path $\ell_{xy}$ in $\eta$ connecting $x$ and $y$, and consisting of edges of $G$. Since $\eta$ contains no infinite contours, $\ell_{xy}$ crosses only finite contours. Moreover, $\ell_{xy}$ crosses finitely many finite contours in total; let $C_1,\ldots, C_k$ be all the finite contours crossed by $\ell_{xy}$. Note that for $1\leq i\leq k$, $\RR^2\setminus C_i$ has exactly one unbounded component. Moreover, this unbounded component of $\RR^2\setminus C_i$ contains both $x$ and $y$. Since any $C_i$ ($1\leq i \leq k$) or $C\in\mathcal{C}_1\cup\mathcal{C}_2$ cannot separate $x$ from $y$ in $\RR^2$, then $x$ and $y$ are in the same unbounded component of $\RR^2\setminus [\cup_{i=1}^k C_i]\cup\cup_{C\in{\mathcal{C}_1\cup\mathcal{C}_2}}C]$. Indeed, $x$ and $y$ are in the same unbounded component of $\eta\setminus [\cup_{i=1}^k C_i]$. By \Cref{l24}, $x$ and $y$ are in the same infinite cluster. But this is impossible. The contradiction implies that $\eta$ contains exactly one infinite cluster.
\end{proof}

\begin{lemma}\label{eoi}Assume that for every infinite contour  in $\phi$ has at most two ends. If $1\leq \sN<\infty$, then the total number of infinite clusters in $\sN+1$.
\end{lemma}
\begin{proof} By \Cref{lc}, $\mathcal{C}_1\cup\mathcal{C}_2$ contains all the infinite contours. Note that every infinite cluster must lie in an unbounded component of $\RR^2\setminus [\cup_{C\in\mathcal{C}_1\cup\mathcal{C}_2}C]$.
By \Cref{bf}, the number of infinite clusters is equal to the number of unbounded components of $\RR^2\setminus [\cup_{C\in\mathcal{C}_1\cup\mathcal{C}_2}C]$. When $1\leq \sN<\infty$, the number of unbounded components in $\RR^2\setminus [\cup_{C\in\mathcal{C}_1\cup\mathcal{C}_2}C]$ is exactly $\sN+1$ by \Cref{l33,eo}.
\end{proof}

\medskip
\noindent\textbf{Proof of \Cref{m1}.} By (A2), either $\mu(\sA)=0$ or
$\mu(\sA)=1$. Assuming that $\mu(\sA)=1$, we will obtain a contradiction.

By (A2), there exists an integer $1\leq k< \infty$, such that $\mu$-a.s.
there exist $k$ infinite 1-clusters. By (A3), $\mu$-a.s.\ there exist $k$
infinite 0-clusters as well. Therefore the total number of infinite clusters
is $\mu$-a.s.\ even; moreover, $\mu$-a.s.\ there exists an infinite 0-cluster
and an infinite 1-cluster. By \Cref{io}, $\mu$-a.s.\ there exist infinite
contours. By \Cref{att,lc}, $\mathcal{C}_1\cup\mathcal{C}_2\neq \emptyset$.

If $\mathcal{C}_2=\emptyset$, by \Cref{bf}, there exists exactly one infinite
cluster; but this is impossible since the total number of infinite clusters
in $\mu$-a.s.\ even.

If $\mathcal{C}_2\neq \emptyset$, by Part II of \Cref{eo} and the assumption
that $\mu(\sA)=1$, we have $1\leq \sN=|\mathcal{C}_2|<\infty$. By (A1) (A2),
$\mu$-a.s.\ the number of primal contours in $\mathcal{C}_2$ and dual
contours in $\mathcal{C}_2$ are equal. Hence  $\sN\geq 2$ is even $\mu$-a.s.
By \Cref{eoi}, the total number of infinite clusters is $\sN+1$, which is
odd. But this is a contradiction to the fact that the total number of
infinite clusters is $\mu$-a.s.\ even.

Now we have completed the proof that $\mu(\sA)=1$ and $\mu$-a.s.\
$|\mathcal{C}_2|$ is 0 or $\infty$. By \Cref{lc}, if $\mathcal{C}_0\neq
\emptyset$, then there is a unique infinite contour; moreover, this infinite
contour lies in $\mathcal{C}_0$. By (A1) and (A2), $\mu$-a.s.\ the number of
primal contours in $\mathcal{C}_0$ and dual contours in $\mathcal{C}_0$ are
equal. Hence  $\mu$-a.s.\ the number of infinite contours in $\mathcal{C}_0$
is even. The contradiction implies that $\mu$-a.s.\
$\mathcal{C}_0=\emptyset$.

The fact that $\mu$-a.s.\ $\mathcal{C}_1\cup\mathcal{C}_2$ contains all the
infinite contours follows from \Cref{att,ec2}. \hfill$\Box$

\section{Uniqueness of the Infinite Primal/Dual Contour under the Finite Energy Assumption}\label{pm4}

In this section, we prove \Cref{m4}.  Let $\mu$ be a probability measure on $\Omega$ satisfying Assumptions (A1), (A2) and (A4).

Let $B_{m,n}^*$ be an $m\times n$ box of $\LL_2$. Assume that the location of $B_{m,n}^*$ satisfies the following condition:
\begin{itemize}
\item $B_{1,1}^*$ consists of the square face of $\LL_2$ containing the origin.
\item when $n$ is odd (resp.\ even), $B_{m,n+1}^*$ is obtained from $B_{m,n}^*$ by adding one column of $m$ square faces of $\LL_2$ to the right (resp.\ left) of $B_{m,n}^*$.
\item when $m$ is odd (resp.\ even), $B_{m+1,n}^*$ is obtained from $B_{m,n}^*$ by adding one row of $n$ square faces of $\LL_2$ to the bottom (resp.\ top) of $B_{m,n}^*$.
\end{itemize}
  Let $B_{m-1,n-1}$ be the interior dual graph of $B_{m,n}^*$, i.e., $B_{m-1,n-1}$ is the subgraph of $\LL_1$ enclosed by $\partial B_{m,n}^*$.

\begin{lemma}\label{sp}Let $\phi$ be a primal contour configuration. Assume that the boundary $\partial B_{3,3}^*$ crosses present edges in $\LL_1$ an even number of times. Then we can change states of edges in $B_{2,2}$ in such a way that all the present edges in $\LL_1$ crossing $\partial B_{3,3}^* $ are in the same contour, and each vertex of $B_{2,2}$ has an even number of incident present edges.
\end{lemma}

The proof of \Cref{sp} involves straightforward (but tedious) case analysis - see \Cref{p51}.

\begin{lemma}\label{spm}Let $\phi$ be a primal contour configuration. Assume that $\partial B_{M,N}^*$ crosses present edges of $\LL_1$ an even number of times. When $M\geq 3$, $N\geq 3$, we can change states of edges of $B_{M-1,N-1}$, in such a way that all the present edges of $\LL_1$ crossing $\partial B_{M,N}^*$ are in the same contour, and each vertex of $B_{M-1,N-1}$ has an even number of incident present edges.
\end{lemma}
\begin{proof}We prove the lemma by induction on $M,N$. First of all, the lemma holds when $M=N=3$ by \Cref{sp}.

Assume that the lemma is true when $M=k$, $N=l$, where $k,l\geq 3$. Now we prove the lemma when $M=k+1$ and $N=l$.

We use $\widehat{B}_{k,l}^*$ to denote the identical copy of $B_{k,l}^*$ obtained from $B_{k+1,l}^*$ by removing the left column. The boundary $\partial B_{k,l}^*$ consists of four line segments $L_N$ (the north boundary), $L_S$ (the south boundary), $L_{W}$ (the west boundary) and $L_{E}$ (the east boundary), such that $L_{N}\cup L_{S}\cup L_{E}\subset \partial B_{k+1,l}^*$. We change states of edges of $\LL_1$ enclosed by $\partial B_{k+1,l}^*$, following the steps below.
\begin{enumerate}[label=\Roman*.]
\item Make all the edges of $B_{k,l-1}$ absent.
\item For each present horizontal edge $e$ of $\LL_1$ crossing $\partial B_{k+1,l}^*\setminus[L_N\cup L_S\cup L_E]$, let $v$ be the endpoint of $e$ within $B_{k+1,l}^*$ and $f$ be the other horizontal edge with endpoint $v$. Make $f$ present.
\item Let $p$ and $q$ be the two corner vertices of $B_{k,l-1}$ in $B_{k+1,l}^*\setminus \widehat{B}_{k,l}^*$. If after the above two steps, both $p$ and $q$ have an odd number of incident present edges, make all the edges along the line segment $[p,q]$ present.
\item Now we consider the case that after the first two steps, exactly
    one of $p$ and $q$ has an odd number of incident present edges; or
    equivalently, exactly one vertical edge incident to $p$ or $q$
    crossing $\partial B_{k+1,l}^*$ is present originally. Let $e$ be the
    present vertical edge of $\LL_1$ crossing $\partial
    B_{k+1,l}^*\setminus[L_N\cup L_S\cup L_E]$. Let $p$ be the endpoint
    of $e$ within $B_{k+1,l}^*$. Let $f_1$ (resp.\ $f_2$) be the
    horizontal edge of $\LL_1$ with endpoint $p$ on the left (resp.
    right) of $e$. If $f_1$ is absent after step II., make $f_2$ present.
    If $f_1$ is present after step II., let $f_3$ be the other vertical
    edge of $\LL_1$ with endpoint $p$; make $f_3$ present. Let $u$ be the
    other endpoint of $f_3$, let $f_4$ (resp.\ $f_5$) be the horizontal
    edge with endpoint $u$ on the left (resp.\ right) of $f_3$. If $f_4$
    is absent, make $f_5$ present. If $f_4$ is present, make $f_5$
    absent.
\end{enumerate}
After the configuration changing process described above, we obtain a configuration satisfying the following conditions.
\begin{enumerate}[label=\roman*.]
\item Each vertex of $\LL_1$ in $B_{k+1,l}^*\setminus \widehat{B}_{k,l}^*$ has an even number of incident present edges.
\item Each present edge in $\LL_1$ crossing $\partial B_{k+1,l}^*$ is in the same contour as a present edge in $\LL_1$ crossing $\partial \widehat{B}_{k,l}^*$.
\item The boundary $\partial\widehat {B}_{k,l}^*$ crosses present edges in $\LL_1$ an even number of times.
\end{enumerate}
See Figure \ref{fig:mg} for an illustration of such a configuration changing process.

\begin{figure}
\subfloat[\textrm{Original}]{\includegraphics[width=.3\textwidth]{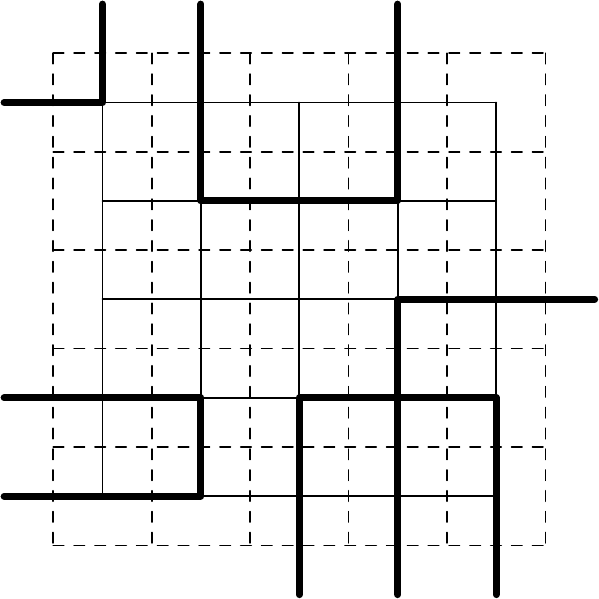}}\qquad\qquad\qquad\qquad
\subfloat[\textrm{Step} 1]{\includegraphics[width = .3\textwidth]{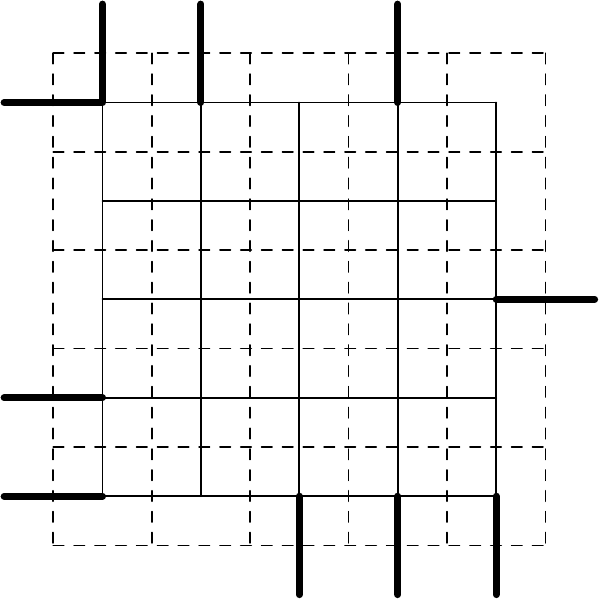}}\\
\subfloat[\textrm{Step} 2]{\includegraphics[width = .3\textwidth]{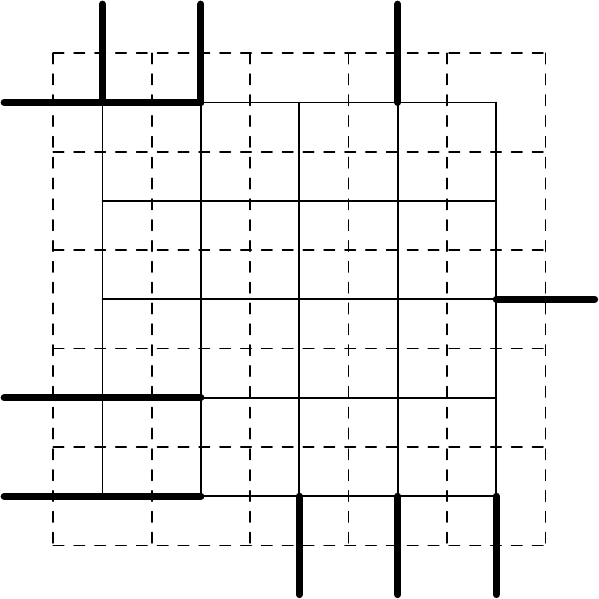}} \qquad\qquad\qquad\qquad
\subfloat[\textrm{Step} 4]{\includegraphics[width = .3\textwidth]{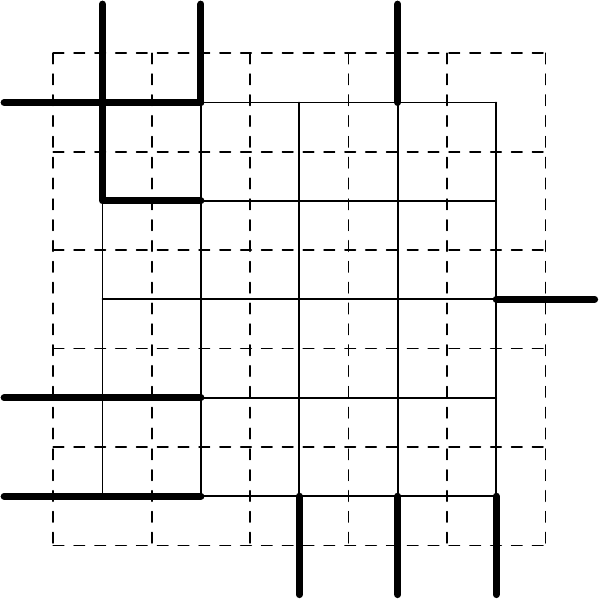}}
\caption{Merging finitely many infinite primal contours into one infinite primal contour: from $B_{k+1,l}^*$ to $B_{k,l}^*$. Dashed lines represent $B_{k+1,l}^*$; solid lines represent $B_{k,l-1}$; thick solid lines represent primal contours.}
\label{fig:mg}
\end{figure}

By induction hypothesis, we can change states of edges enclosed by $\partial \widehat{B}_{k,l}^*$ such that all the present edges of $\LL_1$ crossing $\partial \widehat{B}_{k,l}^*$ are in the same contour, and every vertex of $\LL_1$ enclosed by $\partial \widehat{B}_{k,l}^*$ have an even number of incident present edges. Combining with conditions i.--iii., we infer that the lemma is true when $M=k+1$, $N=l$. The case $M=k$, $N=l+1$ can be proved analogously.
\end{proof}

\begin{lemma}\label{l44}Let $M\geq 3,N\geq 3$. If two primal contour configurations on $B_{M-1,N-1}$ have the same configuration on the boundary (i.e., present and absent edges crossing $\partial B^*_{M,N}$ are the same in the two configurations), then one configuration can be obtained from the other configuration by changing configurations on squares of $B_{M-1,N-1}$. Here by changing configurations on a square, we mean making every present edge around a square face of $\LL_1$ absent, and making every absent edge around a square face of $\LL_1$ present.
\end{lemma}
\begin{proof}Each $\{0,1\}^{V(\LL_2)}$ configuration corresponds to a primal contour configuration. Two neighboring vertices on $V(\LL_2)$ have different states if and only if the edge of $\LL_2$ joining the two vertices crosses a primal contour. Changing a vertex state in $V(\LL_2)$ from 0 to 1, or from 1 to 0, corresponds to changing contour configuration on the corresponding square in $\LL_1$. It is clear that for any two $\{0,1\}^{B_{M,N}^*\cap V(\LL_2)}$ configurations with the same boundary condition on $\partial B_{M,N}^*$, one can be obtained from the other by changing vertex states on finitely many vertices in $[B_{M,N}^*\setminus\partial B_{M,N}^*]\cap V(\LL_2)$.
\end{proof}

\begin{lemma}\label{l61}Let $\mu$ be a probability measure on $\Omega$ satisfying Assumptions $(Ak1)$, $(Ak2)$ and $(A4)$. Then $\mu$-a.s.\ the number of infinite primal contours is 0,1, or $\infty$.
\end{lemma}

\begin{proof}The proof is inspired by \cite{ns81,ZL12m}.

The boundary $\partial B_{M,N}^*$ consists of four line segments. We claim that primal contours cross $\partial B_{M,N}^*$ an even number of times. To see why this is true, note that any configuration of primal contours induces a site configuration $\rho$ in $\{0,1\}^{V(\LL_2)}$, such that for each pair of adjacent vertices $u,v\in V(\LL_2)$, $u$ and $v$ have different states in $\rho$ if and only if the edge $\langle u,v\rangle$ of $\LL_2$ crosses a primal contour. Winding around $\partial B_{M,N}^*$, the states of vertices change an even number of times, and therefore $\partial B_{M,N}^*$ crosses primal contours an even number of times.

Let $t$ be an integer satisfying $1\leq t<\infty$, and let $E_t$ be the event that there exist exactly $t$ infinite primal contours. By (Ak2), either $\mu(E_t)=0$ or $\mu(E_t)=1$.

Assume there exists $t\geq 2$, such that $\mu(E_t)=1$.  Let $F_n$ be the event that $B_{n,n}^*$ intersects all the $t$ infinite primal contours, then
\begin{eqnarray*}
\lim_{n\rightarrow\infty}\mu(F_n)=\lim_{n\rightarrow\infty}\mu(\cup_n F_n)=1.
\end{eqnarray*}
Hence there exists $N\geq 2$, such that
\begin{eqnarray}
\mu(F_N)>\frac{1}{2}>0.\label{fp}
\end{eqnarray}

Assume that $F_{N}$ occurs. Then by \Cref{spm,l44}, after changing configurations on finitely many squares of $B_{N-1,N-1}$ in such a way as described in (A4), we obtain a configuration with exactly one infinite contour. By (\ref{fp}) and (A4), with strictly positive probability, there is exactly one infinite contour. But this is a contradiction to the assumption that $\mu(E_t)=1$, for some $\infty>t\geq 2$. Therefore we conclude that $\mu$-a.s.\ the number of infinite primal contours is 0,1, or $\infty$.
 \end{proof}

\medskip
\noindent\textbf{Proof of Theorem \ref{m4}.}  By \Cref{l61}, it suffices to prove that $\mu$-a.s.\ the number of infinite primal contours is finite.

By \Cref{att},   $\mu$-a.s.\ no contour has more than two ends.

Let $E_{\infty}$ be the event that there exist infinitely many infinite primal contours. By (Ak2), either $\mu(E_{\infty})=0$ or $\mu(E_{\infty})=1$.

Assume that $\mu(E_{\infty})=1$. Then when $N$ is sufficiently large, the $N\times N$ box $B_{N}^*$ of $\LL_2$ centered at the origin crosses at least three infinite contours. By \Cref{spm,l44} and (A5), with positive probability, there exists an infinite contour with at least three ends. The contradiction implies $\mu(E_{\infty})=0$. $\hfill\Box$

\section{Nonexistence of Infinite Clusters under the Finite Energy Assumption}\label{pf23}

In this section, we prove \Cref{m2}. Let $\mu$ be a
probability measure on $\Omega$ satisfying Assumptions
(A1)--(A4).

\begin{proof}By \Cref{m4}, $\mu$-a.s.\ there exists at most one infinite primal contour and at most one infinite dual contour. Recall that $\mathcal{C}_i$ ($i=0,1,2$) consists of all infinite (primal and dual) contours satisfying the condition that the complement of the infinite contour in $\RR^2$ has exactly $i$ unbounded components. By (A1)--(A2), almost surely one of the following cases is true:
\begin{enumerate}
\item $\mathcal{C}_2=\emptyset$ and $\mathcal{C}_1=\emptyset$;
\item $\mathcal{C}_2=\emptyset$ and $\mathcal{C}_1$ consists of exactly one infinite primal contour and one infinite dual contour;
\item $\mathcal{C}_1=\emptyset$ and $\mathcal{C}_2$ consists of exactly one infinite primal contour and one infinite dual contour.
\end{enumerate}

Let $\mathcal{B}$ be the event that there exist infinite clusters. By (A2), either $\mu(\mathcal{B})=0$ or $\mu(\mathcal{B})=1$. Assuming that $\mu(B)=1$, we will obtain a contradiction.

By \Cref{m1}, if $\mu(\mathcal{B})=1$, then $\mu$-a.s.\ there are infinitely
many infinite clusters. By (A3), $\mu$-a.s.\ there are both infinite
0-clusters and infinite-1 clusters. By \Cref{io}, $\mu$-a.s.\ there exist
infinite contours.

By \Cref{att,ec2}, $\mu$-a.s.\ all the infinite contours are in $\mathcal{C}_0\cup\mathcal{C}_1\cup\mathcal{C}_2$. Given that $\mu(\mathcal{B})=1$, $\mathcal{C}_0=\emptyset$. Hence all the infinite contours are in $\mathcal{C}_1\cup\mathcal{C}_2$. Therefore Case I is impossible.

If Case II occurs, by \Cref{uic,bf}, $\mu$-a.s.\ there exists a unique infinite cluster. But this is a contradiction to the fact that there exists infinitely many infinite clusters $\mu$-a.s.

If Case III occurs, by \Cref{eoi}, $\mu$-a.s.\ there exist exactly three infinite clusters. Again this is a contradiction to the fact that there exists infinitely many infinite clusters $\mu$-a.s.

As a result, $\mu(\mathcal{B})$=0. The fact that there are no infinite contours $\mu$-a.s.\ follows from \Cref{ncl}.
\end{proof}

\section{Non-uniqueness of Infinite 1-Clusters}\label{pf1w}

In this section, we prove a few combinatorial and probabilistic results in preparation to prove the remaining main theorems of the paper.

\begin{lemma}\label{icazo}Let $\omega\in\Omega$. Assume that there is exactly one infinite 0-cluster and exactly one infinite 1-cluster in $\omega$. Let $\phi$ be the corresponding contour configuration. Then for each infinite contour $C$, the complement $G\setminus C$ has at most two infinite components.

\end{lemma}

\begin{proof}Let $C$ be an infinite contour in $\phi$ whose complement in $G$ has at least three infinite components. By \Cref{icic}, there exist at least three infinite clusters, contradicting the assumption.
\end{proof}

\begin{lemma}\label{cac}Let $\omega\in \Omega$. If there is exactly one infinite 0-cluster and exactly one infinite 1-cluster in $\omega$, then there exists an infinite contour that is incident to both the infinite 0-cluster and the infinite 1-cluster in $\omega$.
\end{lemma}

\begin{proof}By \Cref{io}, there exist infinite contours.  By \Cref{icazo}, all infinite contours lie in $\mathcal{C}_0\cup\mathcal{C}_1\cup\mathcal{C}_2$. But since there are infinite contours, $\mathcal{C}_0=\emptyset$.

By \Cref{eoi}, since the total number of infinite clusters is two, $\sN=|\mathcal{C}_2|=1$.

Let $C$ be the infinite contour in $\mathcal{C}_2$. By \Cref{icic}, each one of the two unbounded components of $G\setminus C$ contains one infinite cluster incident to $C$. Hence $C$ is adjacent to both the infinite 0-cluster and the infinite 1-cluster.
\end{proof}

\begin{samepage}
\begin{lemma}\label{pnp}Let $\xi_0$, $\xi_1$ be two distinct infinite clusters. Let $C_1$, $C_2$ be two distinct infinite contours. Then it is not possible that the following two conditions happen simultaneously.
\begin{enumerate}[label=(\alph*)]
\item The infinite contour $C_1$ is incident to both $\xi_0$ and $\xi_1$.
\item The infinite contour $C_2$ is incident to both $\xi_0$ and $\xi_1$.
\end{enumerate}
\end{lemma}
\end{samepage}
\begin{proof}We will prove the lemma by contradiction.

Assume that both (a) and (b) occur. We can find points $x\in C_1$ and $y\in C_2$, such that $x$ and $y$ are connected by a path $\ell_{xy}$, consisting of edges of $G$ and two half-edges, (one starting at $x$ and one ending at $y$), such that every vertex of $\ZZ^2$ along $\ell_{xy}$ is in $\xi_1$. Similarly, we can find a point $z\in C_1$ and $w\in C_2$, such that $z$ and $w$ are connected by a path $\ell_{zw}$, consisting of edges of $G$ and two half-edges, (one starting at $z$ and one ends at $w$), such that  every vertex of $\ZZ^2$ along $\ell_{zw}$ is in $\xi_0$. Moreover, we can find a path $\ell_{zx}\subseteq C_1$ connecting $z$ and $x$ and $\ell_{wy}\subseteq C_2$ connecting $w$ and $y$. Viewed as subsets of $\RR^2$, the four paths $\ell_{xy}$, $\ell_{wy}$, $\ell_{zw}$ and $\ell_{zx}$ are disjoint except for the endpoints. Therefore their union is a simple closed curve in $\RR^2$. Let $R\subseteq \RR^2$ be the bounded region enclosed by the curve; see Figure \ref{cci}.

\begin{figure}
 \centering
  \includegraphics[width=0.55\textwidth]{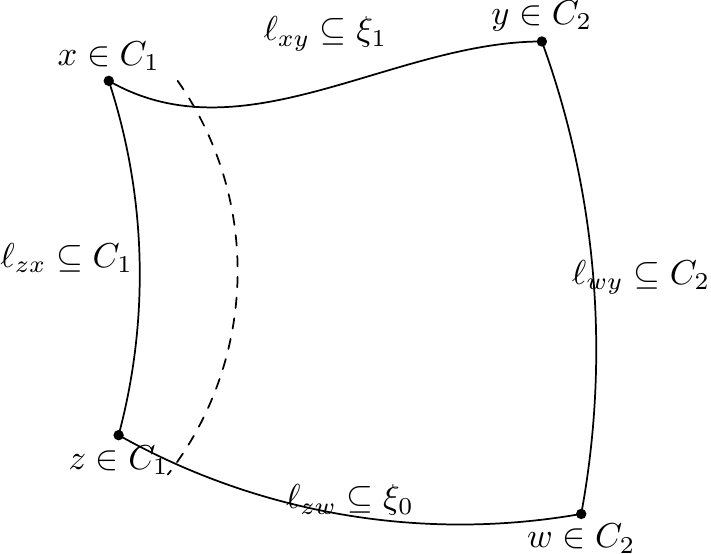}
 \caption{Infinite clusters and incident contours}\label{cci}
 \end{figure}

Let $x_1$ be the first vertex of $G$ along $\ell_{xy}$ starting from $x$; and
let $z_1$ be the first vertex in $\ZZ^2$ along $\ell_{zw}$ starting from $z$.
Let $x_2$ (resp.\ $z_2$) be the midpoint of the line segment $[x,x_1]$
(resp.\ $[z,z_1]$). Since $x\in C_1$ and $x_1\in \xi_1$, the interface of
$C_1$ contains $x_2$. Similarly the interface of $C_1$ contains $z_2$ as
well.

We claim that $x_2$ and $z_2$ are in the same component of the interface of $C_1$. To see why this is true, consider the connected component $\gamma$ of the interface of $C_1$ containing $x_2$; $\gamma$ is either a self-avoiding cycle or a doubly-infinite self-avoiding path. Therefore $\gamma$ crosses $\partial R=\ell_{xy}\cup \ell_{zw}\cup\ell_{zx}\cup \ell_{wy}$ an even number of times. But the only other possible crossing of $\gamma$ with $\partial R$ is $z_2$, therefore $z_2\in \gamma$. Indeed, $\gamma$ cannot cross $C_1$ or $C_2$ because interfaces and contours cannot cross; moreover, $\gamma$ cannot cross $\ell_{xy}$ at a point other than $x_2$ because if that occurs, an edge of $\ell_{xy}$ joins two vertices in different clusters, which is impossible; similarly, $\gamma$ cannot cross $\ell_{zw}$ at a point other than $z_2$.
By similar reasoning any other component of the interface of $C_1$ does not cross $\partial R$.

Recall that $F_{\gamma}$ is the set of vertices whose $\ell^{\infty}$ distance to $\gamma$ is $\frac{1}{4}$. By \Cref{rl}, all those vertices lie in the same cluster. Then $x_1$ and $z_1$ are in the same cluster of the constrained site configuration on $\ZZ^2$. However $x_1\in \xi_1$, $z_1\in \xi_0$, and $\xi_1$ and $\xi_0$ are distinct clusters. This is a contradiction, since they lie in distinct clusters.
\end{proof}

\begin{lemma}\label{cac1}Let $\omega\in \Omega$. If there is exactly one infinite 0-cluster $\xi_0$ and exactly one infinite 1-cluster $\xi_1$ in $\omega$, then the total number of infinite contours incident to  both $\xi_0$ and $\xi_1$ is exactly one.
\end{lemma}
\begin{proof}By \Cref{cac}, there exists at least one infinite contour that is incident to both the infinite 0-cluster and the infinite 1-cluster in $\omega$. By \Cref{pnp}, the number of infinite contours incident to both the infinite 0-cluster and the infinite 1-cluster is at most one. Hence there exists exactly one infinite contour that is incident to both $\xi_0$ and $\xi_1$.
\end{proof}

\section{Unconstrained Percolation}\label{ucpc}

In this section, we introduce certain combinatorial results on unconstrained percolation in $\{0,1\}^{V(\LL_2)}$.

Let $\rho\in \{0,1\}^{V(\LL_2)}$. Let $\psi_1(\rho)\in \Phi_{1}$ be the corresponding contour configuration in which an edge in $\LL_1$ is present if and only if the two vertices whose Euclidean distances to the edge is $1$ have different states. It is trivial to check that each vertex in $V(\LL_1)$ has an even number of incident present edges in $\psi_1(\rho)$. The infinite clusters in $\rho$ and infinite contours in $\Phi_1$ are defined in the usual way. We say a contour in $\psi_1(\rho)$ is \df{incident} to a cluster in $\rho$, if their Euclidean distance is 1.

Let $\psi\in\Phi$. Recall that $\phi^{-1}(\psi)=\{\omega, \theta(\omega)\}\subset \Omega$, where $\theta(\omega)=1-\omega$. Now suppose that $\psi\in\Phi_1$, i.e., $\psi$ has no dual contours. Then in $\omega$, the four sites in $\{2m,2m+1\}\times\{2n+1,2n+2\}$ all have the same state for all $m,n\in\ZZ$. We can associate a dual site configuration $\gamma_2=\gamma_2(\psi)\in\{0,1\}^{V(\LL_2)}$ by setting $\gamma_2(2m+\frac{1}{2},2n+\frac{3}{2})=\omega(2m,2n)$  for all $m,n\in \ZZ$ with probability $\frac{1}{2}$; and setting  $\gamma_2(2m+\frac{1}{2},2n+\frac{3}{2})=1-\omega(2m,2n)$ for all $m,n\in \ZZ$ with probability $\frac{1}{2}$. Each cluster of $\gamma_2$ corresponds via an obvious bijection with a cluster of $\omega$. However, $\gamma_2$ is an \emph{unconstrained configuration}, it may take any value in $\{0,1\}^{V(\LL_2)}$ for suitable $\psi\in\psi_1$. We will use some of our results about constrained percolation to reason about unconstrained configurations.

\begin{proposition}{\label{up}}Let $\rho\in \{0,1\}^{V(\LL_2)}$. If in $\rho$, there is exactly one infinite 0-cluster and one infinite 1-cluster, then in $\psi_1(\rho)$, there is exactly one infinite contour incident to both the infinite 0-cluster and the infinite-1 cluster.
\end{proposition}
\begin{proof}Given the contour configuration $\psi_1(\rho)$ on $\LL_1$,  $\phi^{-1}(\psi_1(\rho))=\{\omega,\theta(\omega)\}$ consists of two constrained configuration in $\Omega$ whose corresponding contour configuration under the mapping $\phi$ is exactly $\psi_1(\rho)$.

By the discussions above, we see that either $\rho=\gamma_2(\psi_1(\rho))$ or $\rho=1-\gamma_2(\psi_1(\rho))$. In either case, in $\omega$ there is exactly one infinite 0-cluster and one infinite 1-cluster. By \Cref{cac1}, in $\phi(\omega)$, there is a unique infinite contour incident to both the infinite 0-cluster and the infinite 1-cluster in $\omega$; hence in $\psi_1(\rho)=\phi(\omega)$, there is a unique infinite contour incident to both the infinite 0-cluster and the infinite 1-cluster in $\rho$.
\end{proof}

\section{Nonexistence of Infinite Clusters in the Marginal Unconstrained Configuration}\label{s8}

In this section, we prove \Cref{m3}.

Let $\mu$ be a measure on $\Omega$ satisfying (A1)--(A4). Let $\omega$ be a
constrained configuration in $\Omega$ with distribution $\mu$. Let
$\psi=\phi(\omega)\in \Phi$ be the corresponding contour configuration. Let
$\psi_1\subseteq\psi$ (resp.\ $\psi_2\subseteq \psi$) be the configuration of
primal (resp.\ dual) contours, i.e., the set of all edges of $\LL_1$ (resp.
$\LL_2$) present in $\phi$. Note that
\begin{eqnarray*}
\psi=\psi_1\cup\psi_2,\qquad\qquad \psi_1\cap\psi_2=\emptyset.
\end{eqnarray*}

Let $\rho=\gamma_2(\psi_1)\in\{0,1\}^{V(\LL_2)}$. Let $E_1$ be the event that  there are no infinite 1-clusters and no infinite 0-clusters in $\rho$.   Then \Cref{m3} follows from \Cref{e2}.

\begin{lemma}\label{l81}Let $C_1,C_2$ be two finite contours. If $C_1$ is in a bounded component of $\RR^2\setminus C_2$, then
\begin{enumerate}
\item $C_2$ is in an unbounded component of $\RR^2\setminus C_1$.
\item each bounded component of $C_1$ is in a bounded component of $C_2$.
\end{enumerate}
\end{lemma}
\begin{proof}II follows from the fact that each bounded component of $C_2$ is a simply-connected set. I follows from II in an obvious way.
\end{proof}

\begin{lemma}\label{l82}Let $C_1,C_2,C_3$ be finite contours. If $C_1$ is in a bounded component of $\RR^2\setminus C_2$, and $C_2$ is in a bounded component of $\RR^2\setminus C_3$, then $C_1$ is in a bounded component of $C_3$.
\end{lemma}
\begin{proof}The lemma follows from \Cref{l81} II.
\end{proof}

\begin{lemma}\label{e2} $\mu(E_1)=1$.
\end{lemma}

\begin{proof}By \Cref{m1}, $\mu$-a.s.\ there are neither infinite contours nor infinite clusters. Let $\mathcal{D}$ be the event that there exist infinitely many finite contours $C$, such that the origin is in a bounded component of $\RR^2\setminus C$. Note that $\mathcal{D}$ is a $\ZZ^2$-translation-invariant event. By (A2), either $\mu(\mathcal{D})=0$ or $\mu(\mathcal{D})=1$.

We first assume that  $\mu(\mathcal{D})=1$. Let $\mathcal{D}_1$ (resp.\ $\mathcal{D}_2$) be the event that there exist infinitely many infinite primal (resp.\ dual) contours $C$, such that the origin is in a bounded component of $\RR^2\setminus C$. By (A1), $\mu(\mathcal{D}_1)=\mu(\mathcal{D}_2)$. Since $\mathcal{D}_1\cup\mathcal{D}_2=\mathcal{D}$, we have
\begin{eqnarray}
\mu(\mathcal{D}_1\cup\mathcal{D}_2)=1.\label{d12}
\end{eqnarray}
By (A2), either $\mu(\mathcal{D}_1)=\mu(\mathcal{D}_2)=1$ or $\mu(\mathcal{D}_1)=\mu(\mathcal{D}_2)=0$. By (\ref{d12}), we have $\mu(\mathcal{D}_1)=\mu(\mathcal{D}_2)=1$. Then in $\rho$, a.s.\ there are no infinite clusters. Hence in this case $\mu(E_1)=1$.

Now we assume that $\mu(\mathcal{D})=0$. Let $\mathcal{C}$ be the collection of all finite contours. By \Cref{m1}, almost surely all the contours are in $\mathcal{C}$. For each $C\in\mathcal{C}$, $\RR^2\setminus C$ has exactly one unbounded component, denoted by $h(C)$. Let $R=\cap_{C\in\mathcal{C}}h(C)$.

First of all, $R$ is connected by \Cref{p34}. Next we will see that $R$ is nonempty and unbounded. Let $\mathcal{C}'$ be the collection of all the finite contours $C$, such that the origin is in a bounded component of $\RR^2\setminus C$. Since $\mu(\mathcal{D})=0$, we have $|\mathcal{C}'|<\infty$.

Assume that $\mathcal{C}'=\{C_1,\ldots, C_k\}$, such that $C_i$ is in a bounded component of $\RR^2\setminus C_{i+1}$, for $1\leq i\leq k-1$. Indeed, $C_1,\ldots, C_k$ can be found in the following way. Let $\ell_0$ be a path from the origin to infinity. Let $C_1$ (resp.\ $C_1$,\ldots $C_k$) be the first (resp.\ second,\ldots, $k$th) contour in $\mathcal{C}'$ intersecting $\ell_0$. By \Cref{l82}, for any integers $i,j$ satisfying $1\leq i<j\leq k$, $C_i$ is in a bounded component of $C_j$.

Let $v_0$ be a vertex in the unbounded component of $\RR^2\setminus C_k$, such that $v_0$ is incident to $C_k$. Then we claim that $v_0\in R$. Indeed, if $v_0$ is in a bounded component of a finite contour $C$, then obviously $C\neq C_k$. Then $C_k$ is in a bounded component of $\RR^2\setminus C$, since $v_0$ is incident to $C_k$, and $v_0$ is in a bounded component of $\RR^2\setminus C$. By \Cref{l82}, for $1\leq i\leq k-1$, $C_i$ is in a bounded component of $C$. Therefore $C\notin\{C_1,\ldots,C_k\}=\mathcal{C}$. Hence the origin is in the bounded component of another finite contour $C$, by \Cref{l81} II. Then $C\in\mathcal{C}$. The contradiction implies that $v_0\in R$. Hence $R$ is non-empty.

Let $\RR\cup\{\infty\}$ be the Riemann sphere. For each $C\in \mathcal{C}$,
let $h_1(C)$ be the component including $\infty$ in
$[\RR^2\cup\{\infty\}]\setminus C$. Then
$\cap_{C\in\mathcal{C}}h_1(C)=R\cup\{\infty\}$, where $R\cup\{\infty\}$ is a
subset of the Riemann sphere. By \Cref{p34}, $R\cup\{\infty\}$ is connected.
Hence $R$ is unbounded. By \Cref{l24} $\mu$-a.s.\ there exists an infinite
cluster in the constrained configuration. This contradicts \Cref{m1}. Hence
$\mu(\mathcal{D})=1$ and $\mu(E_1)=1$.
\end{proof}

\section{Non-symmetric Case}\label{pf4}

In this section, we prove \Cref{m5}.

\subsection{Proof of \Cref{m5}}  By \Cref{m4}, when $\mu$ satisfies (Ak1), (Ak2), (A4), almost surely there exists at most one infinite primal contour. Let $\sE_1$ be the event that there exists a unique infinite primal contour. To prove the theorem, it suffices to prove that $\mu(\sE_1)=0$.

Note that $\sE_1$ is a $2\ZZ\times 2\ZZ$ translation-invariant event. By (A2), we have either $\mu(\sE_1)=0$ or $\mu(\sE_1)=1$.

Assuming that $\mu(\sE_1)=1$, we will derive a contradiction. Let $\phi_1$ (resp.\ $\phi_2$) be the union of all primal (resp.\ dual)  contours. Here we identify $\phi_1$ and $\phi_2$ with their embeddings into the plane.

 Since $\mu(\sE_1)=1$, almost surely there is an unbounded component in $\RR^2\setminus \phi_2$ including the infinite primal contour. Let $\sA_1$ (resp.\ $\sA_0$) be the event that there is an infinite ``1''-cluster (resp.\ ``$0$''-cluster) in $\rho$, where $\rho$ is the induced configuration in $\{0,1\}^{V(\LL_1)}$ as defined in \Cref{ucpc}. By \Cref{l24}, and the bijection between infinite clusters in the constrained configuration and the unconstrained configuration as described in \Cref{ucpc}, we have
  \begin{eqnarray}
\lambda_1(\sA_1\cup \sA_0)=1,\label{a12}
\end{eqnarray}
where $\lambda_1$ is the induced measure on $\{0,1\}^{V(\LL_1)}$ as defined before (A5).
By symmetry we have $\lambda_1(\sA_1)=\lambda_1(\sA_0)$. By (A6), either $\lambda_1(\sA_1)=\lambda_1(\sA_0)=1$, or $\lambda_1(\sA_1)=\lambda_1(\sA_0)=0$. By (\ref{a12}), we have
\begin{eqnarray*}
\lambda_1(\sA_1\cap\sA_0)=1.
\end{eqnarray*}
Namely, in the induced site configuration $\rho$ in $\{0,1\}^{V(\LL_1)}$, almost surely there are infinite 1-clusters and infinite 0-clusters. By (A4), almost surely there is at most one infinite 1-cluster and at most one infinite 0-cluster in $\rho$; by the result in \cite{bk89}. Hence in $\rho$, almost surely there is exactly one infinite 1-cluster $\eta_1$, and exactly one infinite 0-cluster $\eta_0$. Using \Cref{up}, the total number of infinite dual contours of $\phi_2$ incident to both $\eta_1$ and $\eta_0 $ in $\rho$ is exactly one; denote it $C_2$.

Let $C_1$ be the unique infinite primal contour.  Since the vertex set of $C_1$ is a subset of either $\eta_1$ or $\eta_0$, and $C_2$ is incident to both $\eta_1$ and $\eta_0$, there is an infinite cluster in $\omega\in\Omega$ incident to both $C_1$ and $C_2$.

Let $\sB_1$ (resp.\ $\sB_0$) be the event that there is an infinite 1-cluster (resp.\ 0-cluster) in $\omega$ incident to both $C_1$ and $C_2$. Then
\begin{eqnarray}
\mu(\sB_1\cup \sB_0)=1.\label{b12}
\end{eqnarray}
 By (A4), we have $\mu(\sB_1)=\mu(\sB_0)$. By (A2), we have either $\mu(\sB_1)=\mu(\sB_0)=1$, or $\mu(\sB_1)=\mu(\sB_0)=0$. By (\ref{b12}), we have
 \begin{eqnarray*}
 \mu(\sB_0\cap\sB_1)=1.
 \end{eqnarray*}
 But this is impossible by \Cref{pnp}. Hence $\mu$-a.s.\ there are no infinite primal contours when $\mu$ satisfies (Ak1),(Ak2),(A4),(A5).
$\hfill\Box$

\section{Contours and Clusters in Dimer Model and XOR Ising Model}\label{pap}

In this section, we prove \Cref{dc,m17,cxor,uc,nicl,ehel}.

Recall from \Cref{ex} in the introduction that we associate dimer configurations on the square-octagon lattice to each contour configuration on $\LL_1\cup\LL_2$. Recall also that $\mu_D$ is the infinite-volume measure for dimer configurations on the square-octagon lattice obtained from weak limit of measures $\mu_{n,D}$ on tori, defined in (\ref{mnd}). Let $\mu_{D}^{*}$ be the marginal distribution of Type-II edges under $\mu_D$.
In order to prove \Cref{dc}, we first prove a lemma concerning the ergodicity of $\mu_{D}^{*}$.

\begin{lemma}\label{ed}For any given edge weights satisfying $(B1)$, $\mu_{D}^{*}$ is $2\ZZ\times 2\ZZ$-ergodic.
\end{lemma}

\begin{proof}Let $\sR$ be the set of all events that are defined in terms of the states of finitely many Type-II edges and that do not depend on the states of Type-I edges. Let $\sF_2$ be a $\sigma$-algebra on dimer configurations of $G$ generated by $\sR$. Let $E_1,E_2\in \sR$, and let $T_{x}$ and $T_{y}$ be translations by 2 units along horizontal and vertical directions, respectively. By Section 4.4. and 4.5 of \cite{KOS06}, we have
\begin{eqnarray}
\lim_{n\rightarrow\infty}\mu_{D}(E_1\cap T_{i}^n E_2)=\mu_{D}(E_1)\mu_{D}(E_2),\label{mx}
\end{eqnarray}
where $i\in\{x,y\}$.

The measure $\mu_D^*$ is strong mixing by (\ref{mx}), which implies that the measure is totally ergodic. In particular, it is $2\ZZ\times 2\ZZ$-ergodic.
\end{proof}

\bigskip
\noindent\textbf{Proof of \Cref{dc}.} As discussed before, there is a bijection that maps the restriction to Type-II edges of a dimer configuration on the square-octagon lattice, to a constrained percolation configuration in $\Omega$. Each constrained percolation configuration in $\Omega$ induces a contour configuration in $\Phi$. Therefore to prove the theorem, it suffices to prove that under the induced measure $\mu_{D}^*$ on $\Omega$, almost surely there are no infinite clusters. By \Cref{m2}, it suffices to check that $\mu_D^*$ of $\Omega$ satisfies Assumptions (A1)-(A4).

The translation-invariance assumptions (A1) follow from Assumptions (B1), (B4) on dimer weights. The ergodicity assumption (A2) follows from \Cref{ed}.

To check the symmetry assumption (A3), note that switching the states of all the four Type-II edges incident to the same square of the square-octagon lattice will give exactly the same weight; see Figure~\ref{sds}.
\begin{figure}
\centering
\begin{minipage}{.38\textwidth}\centering
\subfloat[$w_{e_1}$]{\includegraphics[width=.41\textwidth]{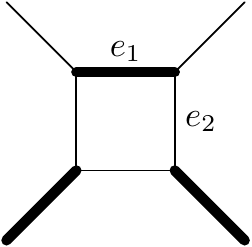}}
\end{minipage}
$\Longleftrightarrow$
\begin{minipage}{.53\textwidth}\centering
\subfloat[$w_{e_1}$]{\includegraphics[width = .3\textwidth]{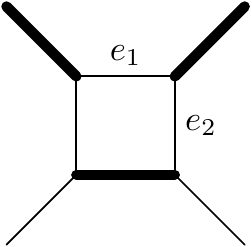}}
\end{minipage}
\begin{minipage}{.38\textwidth}\centering
\subfloat[$w_{e_2}$]{\includegraphics[width = .41\textwidth]{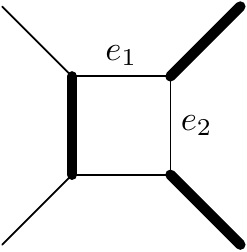}}
\end{minipage}
$\Longleftrightarrow$
\begin{minipage}{.53\textwidth}\centering
\subfloat[$w_{e_2}$]{\includegraphics[width = .3\textwidth]{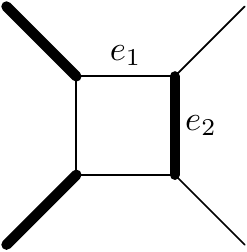}}
\end{minipage}

\begin{minipage}{.38\textwidth}\centering
\subfloat[$1$]{\includegraphics[width = .41\textwidth]{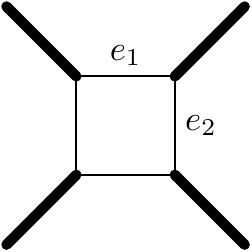}}
\end{minipage}
$\Longleftrightarrow$
\begin{minipage}{.53\textwidth}\centering
\subfloat[$w_{e_1}^2+w_{e_2}^2=1$]{\includegraphics[width = .3\textwidth]{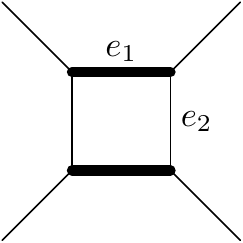}\qquad\includegraphics[width = .3\textwidth]{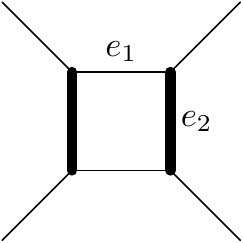}}
\end{minipage}
\caption{Switching configurations for all the four edges incident to the same square gives the same weight. Thick lines represent present edges in a dimer configuration; thin lines represent absent edges.}
\label{sds}
\end{figure}

To check the finite-energy assumption (A4), consider an XOR Ising model with spins located on vertices of $\LL_2$, in which each independent Ising model has coupling constants obtained from the edge weights of the dimer model on the square-octagon lattice by (\ref{we}). It is proved in \cite{bd14} that the configuration of contours of such an XOR Ising model has the same distribution as the configuration of primal contours of the constrained percolation on $G$. Changing a single spin of the XOR Ising model corresponds to changing the states of all the four edges of $\LL_1$ whose dual edges are incident to the spin. The finite-energy assumption (A4) follows from the finite-energy property of the XOR Ising model, which in turn follows from the finite energy of the Ising model. $\hfill\Box$

\bigskip
\noindent\textbf{Proof of \Cref{cxor}.}

Part I: Without loss of generality, we consider a critical XOR Ising model with spins located on vertices of $\LL^2$.  It is proved in \cite{bd14} that the configuration of infinite ``$+$'' or ``$-$'' clusters of the XOR Ising model, have the same distribution as the configuration of infinite components of $\LL_2\setminus\phi_1$, where $\phi_1$ is the (random) union of primal contours of a constrained percolation configurations on $G$, and $\LL_2\setminus\phi_1$ is the subgraph obtained from $\LL_2$ by removing all the edges of $\LL_2$ crossed by $\phi_1$. We claim that for the critical XOR Ising model, the constrained percolation measure satisfies Assumptions (A1)-(A4). To see why that is true, first note that as explained in \Cref{ex}, given dimer edge weights on the square-octagon lattice satisfying (B1)-(B4), by (\ref{we}) we can obtain coupling constants of each independent Ising model in the critical XOR Ising model. Moreover, the coupling constants of any critical Ising model can be obtained in such a way. By the bijection of dimer configurations restricted on Type-II edges and constrained percolation configurations on $\Omega$, if the dimer edge weights satisfy Assumptions (B1)-(B4), then the induced measure on $\Omega$ satisfies (A1)-(A4); see also the proof of \Cref{dc}.

  By \Cref{m3}, almost surely there are no infinite components in $G\setminus \phi_1$.  We conclude that almost surely there are no infinite ``$+$''-clusters, or infinite ``$-$''-clusters in the critical XOR-Ising model.

 Part II: It is proved in \cite{bd14} that the configuration of contours of a critical XOR Ising model with spins located on $\LL_2$ have the same distribution as the configuration primal contours of constrained percolation configurations on $G$, whose probability measures satisfies Assumptions (A1)-(A4). The Part II of the theorem follows from Part I of \Cref{m2}.
 $\hfill\Box$

\bigskip
 \noindent\textbf{Proof of \Cref{uc}.} Let $\nu$ be
the distribution of contours for an non-critical XOR Ising
model. As explained in \Cref{ex}, in the non-critical XOR
Ising model, each independent Ising model has coupling
constants which can be obtained by (\ref{we}) from edge
weights of the dimer model on the square-octagon lattice,
where the dimer edge weights satisfy Assumptions (B1)-(B3).
Using the bijection between dimer configurations on the
square-octagon lattice restricted on Type-II edges and
constrained percolation configurations on $\Omega$, we
infer that $\nu$ is the marginal distribution for primal
contours of a probability measure on the space $\Omega$ of
the constrained percolation configurations, such that the
measure on $\Omega$ satisfies Assumptions (Ak1), (Ak2),
(A4) with $k=1$. Then the theorem follows from \Cref{m4}.\
$\hfill\Box$
\\

In order to prove \Cref{nicl}, we first prove the following lemma.
\begin{lemma}\label{hle}The probability measure for any high temperature XOR Ising model on $\LL_1$ is $2\ZZ\times 2\ZZ$ ergodic.
\end{lemma}

\begin{proof}It suffices to show that the probability
measure for any high-temperature XOR Ising model is strong
mixing. Since each XOR Ising spin is the product of two i.i.d. Ising spins, it suffices to show that the probability measure for any high-temperature Ising model is strong mixing; see Proposition A.1 of \cite{AD15} for a proof.
\end{proof}

\bigskip
\noindent\textbf{Proof of \Cref{nicl}.} Consider a low temperature XOR Ising model with spins located on vertices of $\LL_2$, in which each independent Ising model on $\LL_2$ has coupling constants $J_h$ (resp.\ $J_v$) on horizontal (resp.\ vertical) edges of $\LL_2$, satisfying (\ref{li}). We construct a dimer model on the square-octagon lattice (see Figure \ref{sol}), whose edge weights satisfy the following conditions:
\begin{enumerate}
\item each Type-II edge has weight 1;
\item let $e'$ be a Type-I edge parallel to an edge $e$ of $\LL^2$, such that $e'$ is also an edge of a square face crossed by $e$; then the weight $w_{e'}$ of $e'$ is given by (\ref{we});
\item let $e'$ and $e''$ be two Type-I edges sharing a vertex; then the weights $w_{e'}$ and $w_{e''}$ on $e$ and $e''$ satisfy
\begin{eqnarray*}
w_{e'}^2+w_{e''}^2=1.
\end{eqnarray*}
\end{enumerate}

Given a dimer model on the square-octagon lattice, with edge weights as described above, we can construct another XOR Ising model whose spins are located on vertices of $\LL_1$ as follows. Assume that each independent Ising model has coupling constant $K_h$ (resp.\ $K_v$) on horizontal (resp.\ vertical) edges of $\LL_1$, such that for each edge $f$ of $\LL_1$, let $f'$ be a Type-I edge parallel to $f$, such that $f'$ is also an edge of the square crossed by $f$; then the coupling constant $K_f$ on $f$ and the edge weight $w_{f'}$ of $f'$ satisfy
\begin{eqnarray*}
w_{f'}=\frac{2\exp(-2K_f)}{1+\exp(-4K_f)}.
\end{eqnarray*}
Then we can check that $K_h$, $K_v$ satisfy $F(K_h,K_v)>1$, where $F$ is defined by (\ref{f}). Hence the XOR Ising model on $\LL_1$ is in the high temperature state by (\ref{hii}).

Let $\mu$ be the probability measure of dimer configurations restricted on Type-II edges on the square-octagon lattice, with edge weights given as above, and obtained as the weak limit of probability measures on larger and larger tori. Then $\mu$ is a constrained percolation measure satisfying (A1), (A2), (A4). By \Cref{hle}, $\mu$ also satisfies (A5), since the induced configuration $\rho$ on $\{0,1\}^{V(\LL_1)}$, as given in before (A5), is exactly $\frac{\sigma_2+1}{2}$, where $\sigma_2$ is the high-temperature Ising spin located on $V(\LL_2)$. Hence by \Cref{m5}, in the low-temperature XOR Ising model, almost surely there are no infinite contours.
$\hfill\Box$

 \bigskip
 \noindent\textbf{Proof of \Cref{ehel}.} Without loss of generality, we consider an XOR Ising model with spins located on vertices of $\LL_1$, with probability measure $\mu$.

Consider a $4m\times 4m$ square of $\LL_1$ centered at the origin, denoted by $B_{4m}$. Let $\mu_{1,4m}^{+}$, $\mu_{2,4m}^{+}$ be two independent high temperature Ising measures on $B_{4m}$, both of which has ``$+$'' boundary conditions and the coupling constant $J_h$ (resp.\ $J_v$) on each horizontal (resp.\ vertical) edge, such that $J_h,J_v$ satisfy the high temperature condition (\ref{hii}). Let $\mu_{4m}^{+,+}$ be the XOR Ising measure obtained from the product measure of $\mu_{1,4m}^{+}$ and $\mu_{2,4m}^+$. Let $v$ be a vertex of $\LL_1$ inside the box $B_{4m}$. Let $E_{v}$ be the event that there exists a path in $B_{4m}$ connecting $0$ and $v$ consisting of edges of $\LL_1$, such that every vertex of $\LL_1$ along the path has the state ``$+$'' in the XOR Ising configuration.

Recall that for each vertex $u\in V(\LL_1)$, the state of $u$ in the XOR Ising model satisfy $\sigma_{XOR}(u)=\sigma_1(u)\sigma_2(u)$, where $\sigma_1$, $\sigma_2$ are spins in two independent Ising models, respectively. Conditional on any Ising configuration $\sigma_1$ on $B_{4m}$ with ``+'' boundary conditions, the event $E_{v}$ occurs if and only if there exists a path in $B_{4m}$ connecting $0$ and $v$, consisting of edges of $\LL_1$, such that every vertex of $\LL_1$ along the path has the same state in $\sigma_2$ as its state in $\sigma_1$.

Given the configuration $\sigma_1$, we may modify coupling constants for edges of $\LL_1$ as follows. Let $e\in E(\LL_1)$. If $e$ connects two vertices of $\LL_1$ with the same state in $\sigma_1$, then we preserve the coupling constant of $e$; if $e$ connects two vertices of $\LL_1$ with the opposite states in $\sigma_1$, then we change the coupling constants on $e$ from $J_e$ to $-J_e$. Let $\mu_{2,4m}^{\sigma_1,+}$ be the probability measure for an Ising model on $B_{4m}$ with coupling constants modified according to $\sigma_1$ as described above, and with ``$+$'' boundary conditions. Let $F_{v}$ be the event that there exists a path connecting $0$ and $v$ in $B_{4m}$, consisting of edges of $\LL_1$, such that each vertex of $\LL_1$ has the state ``$+$'' in an Ising model, then
\begin{eqnarray}
\mu_{4m}^{+,+}(E_v|\sigma_1)=\mu_{2,4m}^{\sigma_1,+}(F_v).\label{ccc}
\end{eqnarray}

Consider another Ising model on $\LL_1$, with coupling constant $0$ on all the edges, and external magnetic field $h$ satisfying $0<h<h_0$, where $h_0$ is given by (\ref{h0}). This is equivalent to an i.i.d Bernoulli percolation model on $\LL_1$, in which each vertex of $\LL_1$ is open with probability $p=\frac{e^h}{e^h+e^{-h}}$. When $h<h_0$, we have $p<p_c$, and therefore this is a subcritical percolation model. Let $H_v$ be the event that there exists a path connecting $0$ and $v$, consisting of edges of $\LL_1$, such that every vertex of $\LL_1$ along the path has state ``$+$'' in the Ising configuration. Let $\nu$ be the probability measure for such an Ising model. It is well known that
\begin{eqnarray}
\nu(H_v)\leq C_1e^{-C_2|v|},\label{edc}
\end{eqnarray}
where $C_1>0$, $C_2>0$ are constants independent of $v$, and $|v|$ is the $\ell^1$ distance between $0$ and $v$; since the percolation is subcritical.

Next we show that if $2(J_h+J_v)<h$, then $\nu$, restricted on $B_{4m}$, stochastically dominates $\mu_{2,4m}^{\sigma_1,+}$. It suffices to prove the F.K.G lattice condition, i.e., let $\omega_1,\omega_2\in\{\pm 1\}^{B_{4m}}$, then
\begin{eqnarray}
\nu(\omega_1\vee \omega_2)\mu_{2,4m}^{\sigma_1,+}(\omega_1\wedge\omega_2)\geq \nu(\omega_1)\mu_{2,4m}^{\sigma_1,+}(\omega_2)\label{fkgl}
\end{eqnarray}
where $\omega_1\vee \omega_2$ (resp.\ $\omega_1\wedge\omega_2$) is the maximal (resp.\ minimal) of $\omega_1$ and $\omega_2$. The condition (\ref{fkgl}) is easy to be checked by checking the contribution of every edge of $\LL_1$. Therefore, for the increasing events $F_v\subseteq H_v$, we have
\begin{eqnarray}
\mu_{2,4m}^{\sigma_1,+}(F_v)\leq \nu(F_v)\leq \nu(H_v).\label{sd}
\end{eqnarray}

By (\ref{ccc}), (\ref{edc}), (\ref{sd}), we have
\begin{eqnarray}
\mu_{4m}^{+,+}(E_v)&=&\sum_{\sigma_1\in\{\pm 1\}^{B_{4m}}}\mu_{4m}^{+,+}(E_v|\sigma_1)\mu_{1,4m}^{+}(\sigma_1)\notag\\
&=&\sum_{\sigma_1\in\{\pm 1\}^{B_{4m}}}\mu_{2,4m}^{\sigma_1,+}(F_v)\mu_{1,4m}^{+}(\sigma_1)\notag\\
&\leq&\nu(H_v)\leq C_1e^{-C_2|v|}.\label{edo}
\end{eqnarray}
Since (\ref{edo}) holds for any $m$, letting $m\rightarrow\infty$, we have
\begin{eqnarray*}
\mu(E_v)\leq C_1e^{-C_2|v|}.
\end{eqnarray*}
Then the mean cluster size $\chi$ satisfies
\begin{eqnarray*}
\chi=\sum_{v\in V(\LL_1)}\mu(E_v)<\infty.
\end{eqnarray*}
Then $\mu$-a.s.\ there are no infinite ``$+$''-clusters or infinite
``$-$''-clusters. Using the same arguments as in Proposition 1 of
\cite{Rus78}, and applying \Cref{uc}, we infer that $\mu$-a.s.\ there is
exactly one infinite contour for the high temperature XOR Ising model with
coupling constants $J_h$, $J_v$ satisfying the condition given in the
theorem.

Consider the low temperature XOR Ising model with coupling constants $J_h'$, $J_v'$ obtained from $J_h$, $J_v$ by (\ref{da1}), (\ref{da2}), and with spins located on vertices of $\LL_2$. Since in the high temperature XOR Ising model, almost surely there are infinite contours; in the dual low temperature XOR Ising model, almost surely there are infinite clusters containing the infinite contour in the high XOR Ising model. According to the finite energy condition and translation invariance, there exists at most 1 infinite ``$+$''-cluster and at most 1 infinite ``$-$''-cluster. But if there is 1 infinite ``$+$''-cluster and 1 infinite ``$-$''-cluster, there there exists an infinite contour, but this is a contradiction to \Cref{nicl}. Hence we conclude that in such a low temperature XOR Ising model, almost surely the total number of infinite ``$+$''-clusters and infinite ``$-$'' clusters is exactly 1.
$\hfill\Box$

\bigskip
\noindent\textbf{Proof of \Cref{m17}.}  For any given edge weights satisfying (B1)--(B3), the distributions of the primal and dual contours separating present and absent Type-II clusters of the dimer model on the square octagon lattice are distributions of contours of two dual Ising models on $\LL_2$ and $\LL_1$, respectively. Either both Ising models are critical - then by \Cref{cxor}, almost surely there are no infinite contours; or one Ising model is in the high temperature state, and the other Ising model is in the low temperature state - then by \Cref{uc,nicl}, there exists at most one infinite contour almost surely.
$\hfill\Box$

\begin{appendix}
\section{Proof of \Cref{sp}}\label{p51}

Recall that each contour configuration of $\LL_1$ induces two site
configurations $\rho$ and $1-\rho$ in $\{0,1\}^{V(\LL_2)}$.  An edge $e$ of
$\LL_2$ joins two vertices in $V(\LL_2)$ with different states if and only if
$e$ crosses a contour of $\LL_1$.  Moreover, every site configuration of
$\LL_2$ induces in this way a unique contour configuration of $\LL_1$.

Denote the sets of dual sites
$U:=V(B^*_{3,3})=\{-\tfrac72,-\tfrac32,\tfrac12,\tfrac52\}\times \{-\tfrac52,-\tfrac12,\tfrac32,\tfrac72\}\subset V(\LL_2)$
and $V:=V(B^*_{1,1})=\{-\tfrac32,\tfrac12\}\times\{-\tfrac12,\tfrac32\}\subset U$.  Let $E$ be the set
of primal edges of $\LL_1$ that cross some edge of $B^*_{3,3}$, (i.e.\ that
separate two sites of $U$).  Let $F=E(B_{2,2})\subset E$ (i.e.\ the set of
primal edges that are incident to at least one site in $V$).

To prove the lemma, suppose that we are given a dual site configuration
$\rho$ on $U\setminus V$, which induces a contour configuration $\phi$ on
$E\setminus F$.  We will extend $\rho$ to a configuration $\rho'$ on $U$,
which induces a contour configuration $\phi'$ on $E$, in such a way that all
present edges of $\phi'$ (including those of $\phi$) lie in the same
component.

In fact, our $\phi'$ will have an additional property.  Let $s_1,s_2,s_3,s_4$
be the vertices $(\pm1,0),(0,\pm 1)$ at the centers of the sides of
$B_{2,2}$. We say that the contour configuration $\phi'$ has \df{property
(S)} if it has a component $C$ that contains all of $s_1,\ldots,s_4$. We will
choose $\phi'$ to have property (S), regardless of which vertices $s_i$ had
incident present edges in the original configuration $\phi$.

Let $v_1,v_2,v_3,v_4$ be the primal vertices $(\pm1,\pm1)$ at the corners of
the square $B_{2,2}$, enumerated in counterclockwise order.  Each $v_i$ has
two incident edges in $E\setminus F$. We call $v_i$ a \df{double corner} if
both these edges are present in $\phi$, and a \df{single corner} if exactly
one of them is present.  Provided $\phi'$ is chosen to satisfy property (S),
any single corner $v_i$ will automatically lie in the component $C$. Indeed,
it has exactly two present incident edges in $\phi'$, one of which joins it
to some $s_j$.

Now let $K$ be the number of double corners.  Each $v_i$ is surrounded by
four vertices of $U$ at Euclidean distance $\sqrt 2$.  Let $w_i$ be the one
that lies in $V$, and let $u_i$ be the one opposite $w_i$, which is a corner
of $B^*_{3,3}$.  Suppose $v_i$ is a double corner.  Then we will always set
\begin{equation}\label{opp}
\rho'(w_i)=\rho(u_i)
\end{equation}
This ensures that $v_i$ has all four incident edges present
in $\phi'$. Provided (S) is satisfied, this double corner
will therefore belong to $C$. If $K=3$ or $K=4$ then these
incident edges themselves form a connected set that
contains $s_1,\ldots,s_4$, so (S) is indeed satisfied and
the lemma is proved in these cases.

Given $\rho$, define the \df{parity} of a double corner $v_i$ to be
$(-1)^{i+\rho(u_i)}$.  If all double corners have the same parity, then we
can choose $\rho'$ to be one of the two checkerboard configurations
$\scriptstyle\begin{smallmatrix}0&1\\1&0\end{smallmatrix}$ or
$\scriptstyle\begin{smallmatrix}1&0\\0&1\end{smallmatrix}$ on $V$ in such a
way that \eqref{opp} is satisfied for every double corner.  This ensures that
the origin has four incident present edges, so (S) is satisfied, and the
lemma is proved in this case.  In particular, the cases $K=0$ and $K=1$ are
covered.

The only remaining possibility is that $K=2$, and the two double corners have
different parities.  Modulo symmetries, there are two cases, depending on
whether the two double corners are adjacent or opposite around the square. In
both cases, \eqref{opp} determines $\rho'$ at two vertices of $V$.  We set
the states of $\rho'$ at the other two vertices of $V$ to differ from each
other.  It turns out that condition (S) is then satisfied, as shown in
Figure~\ref{cases}. This completes the proof.
\begin{figure}
\centering
\includegraphics[width=.55\textwidth]{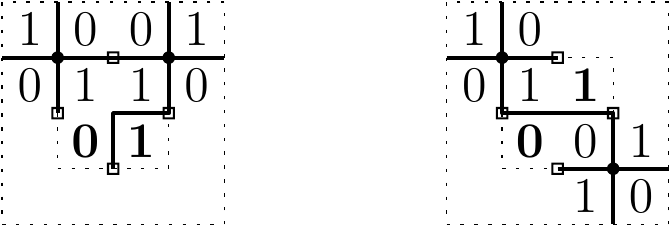}
\caption{The two possible cases of two double corners of different parities.
The sets of sites $V\subseteq U$ are enclosed by the two dotted squares.
The vertices $s_1,\ldots,s_4$ are marked with small squares.
Double corners are marked with filled discs: the four states surrounding a double corner are given by
the configuration $\rho$ together with \eqref{opp}. The final step is to choose the two states shown in
bold.  Contours are shown as solid lines. The unlabelled sites can have
arbitrary states.} \label{cases}
\end{figure}
\end{appendix}
\\
\\
\noindent\textbf{Acknowledgements.} We thank ICERM, the
Newton Institute and the Microsoft Theory Group for
hospitality. We thank Richard Kenyon, Geoffrey Grimmett,
Russel Lyons for helpful discussions and comments. ZL's
research is supported by the Simons Foundation grant 351813
and the National Science Foundation grant 1608896.

\bibliography{cpz18}
\bibliographystyle{amsplain}

\end{document}